\documentclass[12pt, reqno]{amsart}
\usepackage{amsmath,amssymb,amsfonts,amstext, amsthm, amscd}
\usepackage{verbatim}
\usepackage{graphicx}
\usepackage{color}
\usepackage{enumerate}
\usepackage{epstopdf}

\newtheorem{thm}{thm}[section]
\newtheorem{theorem}{Theorem}
\newtheorem{lemma}[thm]{Lemma}
\newtheorem{proposition}[thm]{Proposition}
\newtheorem{corollary}[thm]{Corollary}

\newtheorem{nono-theorem}{Theorem}[]

\newtheorem*{theorem-nonum}{Theorem}
\newtheorem{fact}{Fact}
\newtheorem{claim}[thm]{Claim}

\theoremstyle{definition}
\newtheorem*{definition*}{Definition}

\newcommand{\Z}{{\mathbb Z}}
\newcommand{\R}{{\mathbb R}}
\newcommand{\N}{{\mathbb N}}

\newcommand{\W}{{\mathbb{W}}}

\newcommand{\E}{\mathbb{E}}

\newcommand\B{{\mathcal B}}        

\newcommand\garbage[1]{}

\renewcommand{\P}{{\mathbb P}}

\usepackage{bbm}

\newcommand{\aS}{S_\alpha S}

\begin{document}
\author{Zemer Kosloff}
\thanks{The research of Z.K. was partially supported by ISF grant No. 1570/17}
\address{Einstein Institute of Mathematics,
	Hebrew University of Jerusalem, Edmond J. Safra Campus, Jerusalem 91904,
	Israel}
\email{zemer.kosloff@mail.huji.ac.il}

\author{Dalibor Voln\'y}
\address{Laboratoire de Math\'{e}matiques Raphael Salem,
	UMR 6085, Universit\'e de Rouen Normandie, France}
\email{dalibor.volny@univ-rouen.fr}

	\title{Stable Functional CLT for deterministic systems}


	
\maketitle

\begin{abstract}
We show that $\alpha$ stable L\'evy motions can be simulated by any ergodic and aperiodic probability preserving 
transformation. Namely we show: 

- for $0<\alpha<1$ and every $\alpha$ stable L\'evy motion $\W$, there exists a function $f$ whose partial sum process converges in distribution to $\W$. 

- for $1\leq \alpha <2$ and every symmetric $\alpha$ stable L\'evy motion, there exists a function $f$ whose partial sum process converges in distribution to $\W$,

- for $1< \alpha <2$ and every $-1\leq\beta \leq 1$ there exists a function $f$ whose  associated time series is in the classical domain of attraction of an $S_\alpha (\ln(2), \beta,0)$
random variable.

	\end{abstract}

\section{introduction}
Thouvenot and Weiss showed in \cite{MR2887924} that for every $(X,\B,m,T)$ an aperiodic, probability preserving system and a random variable $Y$, there exists a function $f:X\to\R$ and a sequence $a_n\to\infty$ such that 
\[
\frac{1}{a_n}\sum_{k=0}^{n-1}f\circ T^k\ \ \text{converges in distribution to}\ \ Y.
\]
This result means that any distribution can be approximated by observations of an aperiodic, probability preserving system. See also  \cite{MR3795070} for a refinement of this distributional convergence result for positive random variables and the subsequent \cite{MR3896837} which is concerned with the possible growth rate of the normalizing constants $a_n$. 

A natural question which arises is given a stochastic process $Y=(Y(t)) _{t\in\R}$ whose sample paths are in a Polish space $\mathcal{D}$ can we \textbf{simulate} it using our dynamical system. That is does there exist a measurable function $f:X\to\R$ and normalizing constants $a_n$ and $b_n$ such that the processes $Y_n:X\to \mathcal{D}$ defined by $Y_n(t)(x)=\frac{1}{a_n}\left(\sum_{k=0}^{[nt]}f\circ T^k(x)-b_{[nt]}\right)$ converge in distribution to $Y$.

As noted by Gou\"ezel in \cite{MR3896837}, by a famous result of Lamperti (see \cite[Theorem 8.5.3.]{MR1015093}), any process $Y$ which can be simulated in this manner must be self-similar and the normalizing constants need to be of the form $a_n=n^{\alpha}L(n)$ with $L(n)$ a slowly varying function and $\alpha$ is the self similarity index of the process.  Perhaps due to this, results about simulation of processes are rather scarce; to the best of our knowledge the only such result is \cite{MR1624218}, where the second author has answered a question  of Burton and Denker \cite{MR891642} and showed that every aperiodic, probability preserving system can simulate a Brownian motion with classical normalizing constants $a_n=\sqrt{n}$.  

An important subclass of self-similar process is the class of $\alpha$-stable L\'evy motions which we describe in the next subsection. They include the Brownian motion $(\alpha=2)$ and the Cauchy-L\'evy motion ($\alpha=1$) which is a process with independent increments which are Cauchy-distributed and are often used to model heavy tailed phenomena. 

In this work we show that given an aperiodic, ergodic, probability preserving transformation $(X,\B,m,T)$:
\begin{itemize}
	\item Every $\alpha$-stable L\'evy motion with $\alpha\in(0,1)$ can be simulated by this transformation. 
	\item Every symmetric $\alpha$-stable L\'evy motion can be simulated using this transformation.  
\end{itemize}
One may ask what about general $\alpha$-stable L\'evy motions when $\alpha\in[1,2)$. In this regard we extend the results of \cite{2022arXiv221103448K} and show a classical CLT result for any alpha-stable distribution when $\alpha\neq 1$. 
\subsection{Definitions and statement of the theorems}

A random variable $Y$ is \textbf{stable} if there exists a sequence $Z_1,Z_2,\ldots$ of i.i.d. random variables and sequences $a_n,b_n$ such that
\[
\frac{\sum_{k=1}^n Z_k-a_n}{b_n},\ \text{converges in distribution to} \ Y,\ \text{as}\ n\to\infty. 
\] 
In other words, $Y$ arises as a distributional limit of a central limit theorem, see \cite{MR0322926}. Furthermore in this case, $b_n$ is regularly varying of index $\frac{1}{\alpha}$ which implies that $b_n=n^{1/\alpha}L(n)$ where $L(n)$ is a slowly varying function.  A stable distribution is uniquely defined by its characteristic function (Fourier transform). Namely a random variable is $\alpha$-stable, $0<\alpha\leq 2$, if there exists $\sigma>0$, $\beta\in[-1,1]$ and $\mu\in\R$ such that for all $\theta\in\R$. 
\[
\mathbb{E}(\exp(i\theta Y))=\begin{cases}
	\exp\left\{\left(-\sigma^\alpha|\theta|^\alpha(1-i\beta(\text{sign}(\theta)\tan(\frac{\pi\alpha}{2})\right)+i\mu\theta\right\}, &\alpha\neq 1,\\
	\exp\left\{\left(-\sigma^\alpha|\theta|^\alpha(1+\frac{i\beta}{2}(\text{sign}(\theta)\ln (\theta)\right)+i\mu\theta\right\}, & \alpha=1.
\end{cases}
\]
The constant $\sigma>0$ is the dispersion parameter and $\beta$ is the skewness parameter. In this case we will say that $Y$ is an alpha stable random variable with dispersion parameter $\sigma$, skewness parameter $\beta$ and shift parameter $\mu$, or in short $Y$ is \textbf{an $S_\alpha(\sigma,\beta,\mu)$ random variable}.  If $\mu=\beta=0$ and $\sigma>0$ then the random variable is \textbf{symmetric $\alpha$ stable} and we will abbreviate $Y$ is $S\alpha S(\sigma)$.

A \textbf{probability preserving dynamical system} is a quadruplet $(\mathcal{X},\B,m,T)$ where $(\mathcal{X},\B,m)$ is a standard probability space, $T$ is a measurable self map of $X$ and $m\circ T^{-1}=m$. The system is \textbf{aperiodic} if the collection of all periodic points is a null set.  It is \textbf{ergodic} if every $T$-invariant set is either a null or a co-null set. Given a function $f:X\to \R$ we write $S_n(f):=\sum_{k=0}^{n-1}f\circ T^k$ for the corresponding random walk. 

Recall that if $Y_n$ and $Y$ are random variables taking values in a Polish space $\mathbb{X}$, then $Y_n$ converges to $Y$ in distribution if for every continuous function $G:\mathbb{X}\to\R$,
\[
\lim_{n\to\infty}\E\left(G\left(Y_n\right)\right)=\E(G(Y)). 
\]
Here $\E$ denotes the expectation with respect to the relevant probability measure of the space on which the random variable is defined on. 
\begin{theorem-nonum}(See Theorem \ref{thm: CLT a>1})
For every ergodic and aperiodic probability preserving system $(\mathcal{X},\B,m,T)$, $\alpha>1$ and $\beta\in[-1,1]$, there exists a function $f:X\to\R$ and $B_n\to\infty$ such that
\[
\frac{S_n(f)+B_n}{n^{1/\alpha}}\ \ \text{converges in distribution to} \ S_\alpha(\sqrt[\alpha]{\ln(2)},\beta,0).
\]
\end{theorem-nonum}

A process $\W=\big(\W_s\big)_{s\in [0,1]}$ is an $S_\alpha(\sigma,\beta,0)$ \textbf{L\'evy motion} if it has independent increments and for all $0\leq s<t\leq 1$, $\W_t-\W_s$ is $S_\alpha(\sigma\sqrt[\alpha]{t-s},\beta,0)$ distributed. The existence of an $S_\alpha(\sigma,\beta,0)$ stable motion can be demonstrated via a functional CLT  (also called a weak invariance principle), the details described below appear in \cite{MR2271424}. 

Consider the vector space $D([0,1])$ of functions $f:[0,1]\to\R$ which are right-continuous with left limits, also known as Cadlag functions. Equipped with the Skorohod $J_1$ topology, $D([0,1])$ is a Polish space. Now a natural construction of a distribution on $D([0,1])$ is to take $X_1,X_2,\ldots $, an i.i.d. sequence of random variables and $a_n>0$ and define a $D([0,1])$ valued random variable $\W_n$ via 
\[
\W_n(t)=a_nS_{[nt]}(X)
\]
where $S_n(X):=\sum_{k=1}^nX_k$ and $[\cdot]$ is the floor function. By \cite[Corollary 7.1.]{MR2271424}, if $X_i$ are $S_\alpha(\sigma,\beta,0)$ and $a_n=n^{1/\alpha}$, then $\W_n$ converge in distribution (as random variables on the Polish space $D([0,1])$ with the $J_1$ topology), its limit being an $S_\alpha(\sigma,\beta,0)$ L\'evy motion. The main result of this work is such functional CLT results in the setting of dynamical systems.

\begin{theorem-nonum}
Let  $(\mathcal{X},\B,m,T)$ be an ergodic and aperiodic probability preserving system.
\begin{itemize}
	\item[(Thm \ref{thm: main a<1})] For every $\alpha\in (0,1)$, $\sigma>0$ and $\beta\in[-1,1]$, there exists $f:X\to\R$ such that $\W_n(f)(t):=\frac{1}{n^{1/\alpha}}S_{[nt]}(f)$ converges in distribution to an $S_\alpha(\sigma,\beta,0)$ L\'evy motion. 
	
	\item[(Thm \ref{thm: a>1})] For every $\alpha\in [1,2)$ and $\sigma>0$, there exists $f:X\to\R$ such that $\W_n(f)(t):=\frac{1}{n^{1/\alpha}}S_{[nt]}(f)$ converges in distribution to a $\aS(\sigma)$ L\'evy motion. 
	
\end{itemize} 
\end{theorem-nonum} 

We remark that while the results in Theorems \ref{thm: main a<1} provide a function $f$ whose partial sum process $\W_n(f)$ converges to an $S_\alpha(\sqrt[\alpha]{\ln(2)},\beta,0)$ L\'evy motion, the scaling property of $\alpha$-stable distribution gives that writing $c:=\frac{\sigma}{\sqrt[\alpha]{\ln(2)}}$,  $\W_n(cf)$ converges to an  
$S_\alpha(\sigma,\beta,0)$ L\'evy motion. A similar remark is true with regards to Theorem \ref{thm: a>1}.

\subsection*{Notation} Here and throughout $\log(x)$ denotes the logarithm of $x$ are in base 2, and similarly $\ln(x)$ is the natural logarithm of $x$.

Given two non-negative sequences $a_n$  and $b_n$ we write $a_n\lesssim b_n$ if there exists $C>0$ such that $a_n\leq Cb_n$ for all $n\in\N$ and if in addition $b_n>0$ for all $n$ we write $a_n\sim b_n$ if $\lim_{n\to\infty}\frac{a_n}{b_n}=1$, 

For a function $f:X\to\R$ and $p>0$, $\|f\|_p:=\left(\int |f|^pdm\right)^{1/p}$. 
\section{Construction of the function}
\subsection{Target distributions} \label{sub:target triangular array}
Let $(\Omega,\mathcal{F},\P)$ be a probability space. Let $\left\{X_k(m):\ k,m\in\N\right\}$ be independent random variables so that for every $k\in\N$, $X_k(1),X_k(2),X_k(3),\ldots$ are i.i.d. $S_\alpha\left(\sigma_k,1,0\right)$ random variables with $\sigma_k^\alpha=\frac{1}{k}$. 

For every $k,m\in\N$, define $Y_k(m)=X_k(m)1_{\left[2^k\leq X_k(m)\leq 4^k\right]}$ and its discretisation on a grid of scale $4^{-k}$ defined by 
\[
Z_k(m)=\sum_{j=2^k4^k}^{4^{2k}} \left(\frac{j}{4^k}\right)1_{\left[\frac{j}{4^{k}}\leq Y_k(m)<\frac{j+1}{4^k}\right]}.
\] 
The following fact easily follows from the definitions.

\begin{fact} \label{fact:1}
For every $k\in\N$, $Z_k(1),Z_k(2),\ldots$ are i.i.d. random variables supported on the finite set $\left\{2^k,2^k+4^{-k},\ldots,4^k\right\}$ and for all $m\in\N$, $$0\leq Y_k(m)-Z_k(m)<4^{-k}.$$
\end{fact} 
The construction of the cocycle will hinge on realizing a triangular array of the $Z$ random variables in a dynamical system. 

\subsection{Construction of the function}
Let $(\mathcal{X},\B,m,T)$ be an ergodic, aperiodic, probability preserving system. We first recall some definitions and the copying lemma of \cite{MR4374685} and its application as in \cite{2022arXiv221103448K}.

A finite partition of $\mathcal{X}$ is \textbf{measurable} if all of its pieces (atoms) are Borel-measurable. 
Recall that a finite sequence of random variables $X_1,\ldots,X_n:\mathcal{X}\to \R$, each taking finitely many of values, is \textbf{independent of a finite partition} $\mathcal{P}=(P)_{P\in \mathcal{P}}$ if for all $s\in \R^n$ and $P\in\mathcal{P}$,
\[
m\left(\left(X_j\right)_{j=1}^n=s|P\right)=m\left(\left(X_j\right)_{j=1}^n=s\right).
\]
We will embed the triangular array using the following key proposition. 
\begin{proposition}\cite[Proposition 2]{MR4374685}\label{prop:KV20} Let $(\mathcal{X},\B,m,T)$ be an aperiodic, ergodic, probability preserving transformation and $\mathcal{P}$ a finite-measurable partition of $\mathcal{X}$. For every finite set $A$ and $U_1,U_2,\ldots,U_n$ an i.i.d. sequence of $A$ valued random variables, there exists $f:\mathcal{X}\to A$ such that $(f\circ T^j)_{j=0}^{n-1}$ is distributed as $(U_j)_{j=1}^n$ and $(f\circ T^j)_{j=0}^{n-1}$ is independent of $\mathcal{P}$. 
\end{proposition} 
Using this we deduce the following.

\begin{corollary}\label{cor: weak Sinai}
Let $(\mathcal{X},\B,m,T)$ be an aperiodic, ergodic, probability preserving transformation and $\left(Z_k(j)\right)_{\{k\in\N,1\leq j\leq 4^{k^2} \}}$ be the triangular array from subsection \ref{sub:target triangular array}. There exist functions $f_k,g_k:\mathcal{X}\to\mathbb{R}$ such that $\left(f_k\circ T^{j-1}\right)_{\{k\in\N,1\leq j\leq  4^{k^2} \}}$ and $\left(g_k\circ T^{j-1}\right)_{\{k\in\N,1\leq j\leq  4^{k^2} \}}$ are independent and each is distributed as $\left(Z_k(j)\right)_{\{k\in\N,1\leq j\leq 4^{k^2}\}}$.
\end{corollary} 
\begin{proof}
The sequence $\left(Z_k(m)\right)_{n\in\N,1\leq m\leq 2\cdot4^{k^2}}$ is a sequence of independent random variables and for each $k$, $\left(Z_k(m)\right)_{1\leq m\leq 2\cdot 4^{k^2}}$ are i.i.d. random variables which take finitely many values. 

Proceeding verbatim as in the proof of \cite[Corollary 4]{2022arXiv221103448K}, one obtains a sequence of functions $f_k:\mathcal{X}\to\R$ such that $\left(f_k\circ T^{j-1}\right)_{\{k\in\N,1\leq j\leq  2\cdot 4^{k^2} \}}$ is distributed as $\left(Z_k(j)\right)_{\{k\in\N,1\leq j\leq 2\cdot 4^{k^2}\}}$. Setting $g_k=f_k\circ T^{4^{k^2}}$ the proof is concluded. 
\end{proof}


From now on let $(\mathcal{X},\B,m,T)$ be an aperiodic, ergodic dynamical system and $(f_k)_{k=1}^\infty$ and $(g_k)_{k=1}^\infty$ are the functions from Corollary \ref{cor: weak Sinai}. 
\begin{lemma}\label{lem: BCT 1}
 $\#\{k\in\N: f_k\neq 0\ \text{or}\ g_k\neq 0\}<\infty$, $m$-almost everywhere. 
	\end{lemma}
\begin{proof}
Since $f_k$ and $g_k$ are $Z_k(1)$ distributed and $X_k(1)$ is $S_\alpha(\sigma_k,1,0)$ distributed, it follows from Proposition \ref{prop: ST} that
\begin{align*}
m(f_k\neq 0\ \text{or} \ g_k\neq 0)&\leq m\left(f_k\neq 0\right)+m\left(g_k\neq 0\right)\\
&=2\P\left(Z_k(1)\neq 0\right)\\
&\leq 2\P\left(X_k(1)>2^k\right)\leq C\frac{2^{-\alpha k}}{k},
\end{align*}
where $C$ is a global constant which does not depend on $k$. Using the union bound and stationarity, 
the right hand being summable, the claim follows from the Borel-Cantelli lemma. 
\end{proof}
In what follows, we assume that $\alpha\in(0,2)$ is fixed and $f_k$ and $g_k$ correspond to the functions in Corollary \ref{cor: weak Sinai}. In addition we write for $h:\mathcal{X}\to\R$ and $n\in\mathbb{N}$,
\[
S_n(h):=\sum_{k=0}^{n-1}h\circ T^k.
\] 

Define 
\[
f=\sum_{k=1}^\infty f_k\ \text{and}\ g=\sum_{k=1}^\infty g_k
\]

Note that by Lemma \ref{lem: BCT 1}, $f$ and $g$ are well defined as the sum in their definition is almost surely a sum of finitely many functions. Recall that the (re scaled) partial sum process of  a function $h:\mathcal{X}\to\R$ is, 
\[
\mathbb{W}_n(h)(t)=\frac{1}{n^{1/\alpha}}S_{[nt]}(h),\ \ 0\leq t\leq 1. 
\]

\begin{theorem}\label{thm: main a<1}
Assume $0<\alpha<1$. Fix $\beta\in [-1,1]$ and define 
\begin{align*}
h_k&:=\left(\frac{1+\beta}{2}\right)^{1/\alpha}f_k-\left(\frac{1-\beta}{2}\right)^{1/\alpha}g_k\\
h&:=\left(\frac{1+\beta}{2}\right)^{1/\alpha}f-\left(\frac{1-\beta}{2}\right)^{1/\alpha}g=\sum_{k=1}^\infty h_k.
\end{align*}
$\W_n(h)\Rightarrow^d \W$ where $\W$ is an $S_\alpha(\ln(2),\beta,0)$ L\'evy stable motion. 
\end{theorem}
We also have a FCLT version for general $\alpha\in(0,2)$ when the limit is $S\alpha S$. Recall that the functions $f_k$ and $g_k$ are related by $g_k=f_k\circ T^{4^{k^2}}$. 
\begin{theorem}\label{thm: a>1}
Assume $\alpha\in[1,2)$. Define
\begin{align*}
	h_k&:=f_k-g_k\\
	h&:=f-g=\sum_{k=1}^\infty h_k.
\end{align*}
$\W_n(h)\Rightarrow^d\W$ where $\W$ is a $\aS\left(\sqrt[\alpha]{2\ln(2)}\right)$ L\'evy motion. 

\end{theorem}
\subsection{General CLT for $\alpha>1$}
Recall that a \textbf{coboundary} for a measure preserving transformation is a function $H$ such that there exists a function $G$, called a \textbf{transfer function}, so that $H=G-G\circ T$. The resulting cocycle (sum process) of the  coboundaries $f_k - g_k$ from the proof of Theorem \ref{thm: a>1} converges to a symmetric $\alpha$-stable distribution. To get a skewed $\alpha$-stable limit we thus use a different kind of coboundaries as described below.

Set $D_k:=4^{\alpha k}$,
\[
\varphi_k:=\frac{1}{D_k}\sum_{j=0}^{D_k-1}f_k\circ T^j
\]
and $h_k:=f_k-\varphi_k$. We note that the $h_k$ and $h$ in this subsection denote different functions than the ones in the previous subsection. 
\begin{lemma}\label{lem: h is well defined a>1}
If $\alpha\in(1,2)$, then $\sum_{k=1}^N h_k$ converges in $L^1(m)$ and almost surely as $N\to\infty$. 
\end{lemma}	
\begin{proof}
	By Fubini's theorem it suffices to show that $\sum_{k=1}^\infty\int |h_k|dm<\infty$. 
	
	To that end, for a fixed $k$ we have 
	\begin{align*}
		\int |h_k|dm&\leq \int|f_k|dm+\frac{1}{D_k}\sum_{j=0}^{D_k-1}\int\left|f_k\circ T^j\right|dm
		=2\int|f_k|dm,
	\end{align*}
	where the last equality is true as $T$ preserves $m$. Next $f_k$ and $Z_k(1)$ are equally distributed and 
	\[
	Z_k(1)\leq Y_k(1)\leq X_k(1)1_{\left[X_k(1)\geq 2^k\right]}.
	\]
	As $\alpha>1$, it follows from this and Corollary \ref{cor: moments below a}  that there exists $C>0$ such that for all $k\in\N$,
	\begin{align*}
		\int|f_k|dm&=\E\left(Z_k(1)\right)\\
		&\leq \E\left(X_k(1)1_{\left[X_k(1)\geq 2^k\right]}\right)\leq C\frac{2^{k(1-\alpha)}}{k}. 
	\end{align*}
	We conclude that
	\[
	\sum_{k=1}^\infty\int |h_k|dm\leq C\sum_{k=1}^\infty \frac{2^{k(1-\alpha)}}{k}<\infty. 
	\]
\end{proof}
Following this, we write $h=\sum_{k=1}^\infty h_k$ and throughout this subsection and Section \ref{sec: Skewed CLT a>1}, $h$ always corresponds to this function. Note that for every $k\in\N$, $\E\left(X_k(1)1_{\left[X_k(1)\leq 2^k\right]}\right)$ exists, write 
\[
B_n:=n\sum_{k=\frac{1}{2\alpha}\log(n)}^{\frac{1}{\alpha}\log(n)}\E\left(X_k(1)1_{\left[X_k(1)\leq 2^k\right]}\right).
\]
\begin{theorem}\label{skewed CLT for a>1}
Assume $\alpha\in (1,2)$. 	$\frac{S_n\left(h\right)+B_n}{n^{1/\alpha}}$ converges in distribution to an $S_\alpha(\ln(2),1,0))$ random variable. 
\end{theorem}
The following claim gives the asymptotic of $B_n$.
\begin{claim}
	For every $\alpha\in (1,2)$, there exists $c_\alpha>0$ such that $B_n=c_\alpha n(\log(n))^{1-\frac{1}{\alpha}}(1+o(1))$ as $n\to\infty$.
\end{claim}
\begin{proof}
	Recall that $\sigma_k=k^{-1/\alpha}$. Since $\frac{2^k}{\sigma_k}\to \infty$ as $k\to\infty$, it follows from the monotone convergence theorem that if $Z$ is an $S_\alpha(1,1,0)$ random variable, then
	\[
	\lim_{k\to\infty}\E\left(Z1_{[\sigma_kZ\leq 2^k]}\right)=\E(Z)=:\eta_\alpha>0. 
	\]
	Now for every $k$, $X_k(j)$ and $\sigma_kZ$ are equally distributed. Consequently,
	\[
	\E\left(X_k(1)1_{\left[X_k(1)\leq 2^k\right]}\right)=\sigma_k\E\left(Z1_{\left[\sigma_kZ\leq 2^k\right]}\right)=\sigma_k\eta_\alpha\left(1+o_{k\to\infty}(1)\right). 
	\]
	The claimed asymptotic now follows from this and 
	\[
	\sum_{k=\frac{1}{2\alpha}\log(n)}^{\frac{1}{\alpha}\log(n)}\sigma_k\sim \left(\frac{1}{\alpha}\log(n)\right)^{1-\frac{1}{\alpha}}\left(1-2^{\frac{1}{\alpha}-1}\right),\ \ \text{as}\ n\to\infty. 
	\]
\end{proof}
Now write 
\[
\hat{h}_k:=g_k-\frac{1}{D_k}\sum_{j=0}^{D_k-1}g_k\circ T^j,
\]
and $\hat{h}:=\sum_{k=1}^\infty \hat{h}_k$. Note that $\hat{h}$ is well defined as for all $k$, $\hat{h}_k=h_k\circ T^{4^{k^2}}$ so $h$ is a limit in $L^1$ by Lemma \ref{lem: h is well defined a>1}.  
\begin{theorem}\label{thm: CLT a>1}
	Assume $\alpha>1$. Fix $\beta\in [-1,1]$ and define 
	\[
		H:=\left(\frac{\beta+1}{2}\right)^{1/\alpha}h-\left(\frac{1-\beta}{2}\right)^{1/\alpha}\hat{h}.
	\]
Then	$\frac{1}{n^{1/\alpha}}\left(S_n(H)+B_n\left(\left(\frac{1+\beta}{2}\right)^{1/\alpha}-\left(\frac{1-\beta}{2}\right)^{1/\alpha}\right)\right)$ converges in distribution to  $S_\alpha(\ln(2),\beta,0)$. 
\end{theorem}
\subsection{Strategy of the proof of Theorems \ref{thm: main a<1} and \ref{thm: a>1}}\label{sub: strat}
The proof starts with writing for $\psi\in\{h,f,g\}$

\begin{equation}\label{eq: decomp}
	\W_n(\psi)=\W_n^{(\mathbf{S})}(\psi)+\W_n^{(\mathbf{M})}(\psi)+\W_n^{(\mathbf{L})}(\psi) 
\end{equation}
where 
\begin{align*}
	\W_n^{(\mathbf{M})}(\psi)&:= \sum_{k=\frac{1}{2\alpha}\log(n)+1}^{\frac{1}{\alpha}\log(n)}\W_n(\psi_k)\\
	\W_n^{(\mathbf{S})}(\psi)&:=	\sum_{k=1}^{\frac{1}{2\alpha}\log(n)} \W_n(\psi_k)\\
	\W_n^{(\mathbf{L})}(\psi)&:=\sum_{k=\frac{1}{\alpha}\log n+1}^\infty \W_n(\psi_k).			
\end{align*}
Writing $\left\|\cdot\right\|_\infty$ for the supremum norm, we first show that $\left\|	\W_n^{(\mathbf{S})}(h)\right\|_\infty$ and $\left\|	\W_n^{(\mathbf{L})}(h)\right\|_\infty$ converge to $0$ in probability, hence the two processes converge to the zero function in the uniform (and consequently the $J_1$) topology. 

Next we show that $\W_n^{(\mathbf{M})}(h)$ converges in distribution (in the $J_1$ topology) to the correct limiting process. 

Finally, we use Slutsky's theorem, also known as the convergence together lemma, in the (Polish) Skorohod $J_1$ topology, to deduce the weak convergence result for $\W_n(h)$. 
\begin{lemma}\label{lem: conv together D}
Let $A_n,B_n$ and $W$ be $D[0,1]$ valued processes such that $A_n\Rightarrow^d0$ in the uniform topology and $B_n\Rightarrow W$ in the $J_1$ topology then $A_n+B_n\Rightarrow^dW$ in the $J_1$ topology.  
\end{lemma}
We remark that Lemma \ref{lem: conv together D} follows from \cite[Theorem 3.1.]{MR1700749} and the fact that the uniform topology is stronger than the $J_1$ topology on $D[0,1]$. 


\section{Proof of Theorem \ref{thm: main a<1}}
We carry out the proof strategy as stated in Subsection \ref{sub: strat}. In what follows $(X,\B,m,T)$ is an ergodic, aperiodic probability preserving system, $\beta\in [-1,1]$ and $\alpha\in(0,1)$ is fixed and the functions $f_k$ are as in Theorem \ref{thm: main a<1}. 

This section has 2 subsections, in the first we prove results on $\W_n^{(\mathbf{S})}(f)$, $\W_n^{(\mathbf{M})}$ and $\W_n^{(\mathbf{L})}(f)$. These results combined prove Theorem \ref{thm: main a<1} in the totally skewed to the right ($\beta=1$) case. In the second subsection we show how to deduce Theorem \ref{thm: main a<1} from these results.
\subsection{Case $\beta=1$}
\begin{lemma}\label{lem: a<1 Large}
$\lim_{n\to\infty}m\left(\left\|\W_n^{(\mathbf{L})}(f)\right\|_\infty \neq 0\right)=0$. 
\end{lemma}
\begin{proof}
We have the inclusion 
\[
\left[\left\|\W_n^{(\mathbf{L})}(f)\right\|_\infty \neq 0\right]\subset \bigcup_{k=\frac{1}{\alpha}\log(n)+1}^\infty \bigcup_{j=0}^{n-1}\left[f_k\circ T^j\neq 0\right].
\]
Therefore,
\begin{align*}
m\left(\left\|\W_n^{(\mathbf{L})}(f)\right\|_\infty\neq 0\right) &\leq \sum_{k=\frac{1}{\alpha}\log(n)+1}^\infty \sum_{j=0}^{n-1}m\left(f_k\circ T^j\neq 0\right)\\
&=\sum_{k=\frac{1}{\alpha}\log(n)+1}^\infty n \cdot m\left(f_k\neq 0\right)\\
&\leq C n\sum_{k=\frac{1}{\alpha}\log(n)+1}^\infty \frac{2^{-\alpha k}}{k},
\end{align*}
here the last inequality is from the proof of Lemma \ref{lem: BCT 1}. The result now follows since
\[
n\sum_{k=\frac{1}{\alpha}\log(n)+1}^\infty \frac{2^{-\alpha k}}{k}\lesssim \frac{1}{\log(n)}\xrightarrow[n\to\infty]{}0,
\]
\end{proof}

\begin{lemma}\label{lem: a<1 small}
$\left\|\W_n^{(\mathbf{S})}(f)\right\|_\infty\xrightarrow[n\to\infty]{}0$ in measure. 
\end{lemma}
\begin{proof}

Recall that for all $k\in\N$, $f_k$ is distributed as $Z_k(1)$, whence $f_k\geq 0$  and 
\begin{align*}
 \left\|\W_n^{(\mathbf{S})}(f)\right\|_\infty&=\max_{0\leq t\leq 1}\left| \frac{1}{n^{1/\alpha}}	\sum_{k=1}^{\frac{1}{2\alpha}\log(n)}\left(\sum_{j=0}^{[nt]-1}f_k\circ T^j \right)\right|\\
&= \frac{1}{n^{1/\alpha}}	\sum_{k=1}^{\frac{1}{2\alpha}\log(n)}\left(\sum_{j=0}^{n-1}f_k\circ T^j \right).
\end{align*}

For every $k,j\in\Z$, $f_k\circ T^j$ is distributed as $Z_k(1)$ and $$0\leq Z_k(1)\leq X_k(1)\mathbf{1}_{\left[0\leq X_k(1)\leq 4^k\right]}.$$ 
By Corollary \ref{cor: moments}, there exists $C>0$ such that for all $k,j\in\N$,
\begin{align*}
	\left\|f_k\circ T^j\right\|_{1}&= \E\left(Z_k(1)\right)\\
	&\leq \E\left(X_k(1) \mathbf{1}_{\left[0\leq X_k(1)\leq 4^k\right]}\right)\leq C\frac{4^{k(1-\alpha)}}{k}.
\end{align*} 
Consequently, 

\begin{align*}
\left\|  \left\|\W_n^{(\mathbf{S})}(f)\right\|_\infty\right\|_1 &\leq n^{-\frac{1}{\alpha}}\sum_{k=1}^{\frac{1}{2\alpha}\log(n)}\sum_{j=0}^{n-1}	\left\|f_k\circ T^j\right\|_1\\
&\leq n^{-\frac{1}{\alpha}}\sum_{k=1}^{\frac{1}{2\alpha}\log(n)}Cn\frac{4^{(1-\alpha)k}}{k}\lesssim \frac{1}{\log(n)}\xrightarrow[n\to\infty]{}0.
\end{align*}
A standard application of the Markov inequality shows that $\left\|\W_n^{(\mathbf{S})}(f)\right\|_\infty$ converges to $0$ in measure, concluding the proof.



\end{proof}
The rest of this subsection is concerned with the proof of the following result for $\W_n^{(\mathbf{M})}(f)$. 
 
\begin{proposition}\label{prop: a<1 middle h}
$\W_n^{(\mathbf{M})}(f)$ converges in distribution to $\W$, an $S_\alpha\left(\sqrt[\alpha]{\ln(2)},1,0\right)$ L\'evy motion. 
\end{proposition}

For $V\in\{X,Y,Z\}$ and $k,n\in\N$, define
\[
S_n\big(V_k\big):=\sum_{j=1}^n V_k(j).
\] 
We introduce the following $D[0,1]$ valued processes on $(\Omega,\mathcal{F},\P)$:
\begin{align*}
\W_n^{(\mathbf{M})}(Z)(t)&:=\frac{1}{n^{1/\alpha}}\sum_{k=\frac{1}{2\alpha}\log(n)+1}^{\frac{1}{\alpha}\log(n)}S_{[nt]}\left(Z_k(\cdot)\right),\\ \W_n^{(\mathbf{M})}(Y)(t)&:=\frac{1}{n^{1/\alpha}}\sum_{k=\frac{1}{2\alpha}\log(n)+1}^{\frac{1}{\alpha}\log(n)}S_{[nt]}\left(Y_k(\cdot)\right)\ \ \text{and}\\
\W_n^{(\mathbf{M})}(X)(t)&:=\frac{1}{n^{1/\alpha}}\sum_{k=\frac{1}{2\alpha}\log(n)+1}^{\frac{1}{\alpha}\log(n)}S_{[nt]}\left(X_k(\cdot)\right).
\end{align*}	
The reason for its definition is the following. 
\begin{lemma}\label{lem: a<1 equal dist}
$\W_n^{(\mathbf{M})}(f)$ and $\W_n^{(\mathbf{M})}(Z)$ are equally distributed. 
\end{lemma}
\begin{proof}
By the definition of $f_k$, $\left\{f_k\circ T^{j-1}:\ k\in\N,\ 1\leq j\leq  4^k\right\}$ and $\left\{Z_k(j):\ k\in\N,\ 1\leq j\leq 4^k\right\}$ are equally distributed.  

The  $G_n: \prod_{k\in\N} \R^{2\cdot 4^k}\to D[0,1]$ defined by for all $0\leq t\leq 1$,
\[
G_n\left(\left(x_k(j)\right)_{\ k\in\N,\ 1\leq j\leq 4^k}\right)(t):=\frac{1}{n^{1/\alpha}}\sum_{k=\frac{1}{2\alpha}\log(n)+1}^{\frac{1}{\alpha}\log(n)}\sum_{j=1}^{[nt]}x_k\big(j\big)
\]
is continuous. 

As $G_n\left(\left(f_k\circ T^{j-1}\right)_{\ k\in\N,\ 1\leq j\leq  4^k}\right)=\W_n^{(\mathbf{M})}(f)$ and similarly \\$G_n\left(\left(Z_k(j)\right)_{\ k\in\N,\ 1\leq j\leq 4^k}\right)=\W_n^{(\mathbf{M})}(Z)$, the claim follows from the continuous mapping theorem.  
\end{proof}
Using this equality of distributions, it suffices to show that $\W_n^{(\mathbf{M})}(Z)$  converges in distribution to an $S_\alpha(1,1,0)$ L\'evy motion. This follows from the the following and the convergence together Lemma (Lemma \ref{lem: conv together D}). 

\begin{lemma}\label{lem: a<1 middle Z}
$ $
\begin{itemize}
	\item[(a)] $
\left\| \W_n^{(\mathbf{M})}(X)-\left(\W_n^{(\mathbf{M})}(Z)\right)\right\|_\infty\xrightarrow[n\to\infty]{}0$ in measure. 
\item[(b)] $\W_n^{(\mathbf{M})}(X)$ converges in distribution to an $S_\alpha\left(\sqrt[\alpha]{\ln(2)},1,0\right)$ L\'evy motion.
\end{itemize}
Consequently  $\W_n^{(\mathbf{M})}(Z)$ converges in distribution to an $S_\alpha\left(\sqrt[\alpha]{\ln(2)},1,0\right)$ L\'evy motion
\end{lemma}
\begin{proof}[Proof of Lemma \ref{lem: a<1 middle Z}.(a)]
For every $k,m\in\N$,\footnote{Here note that as $\alpha<1$ a skewed $\alpha$ stable random variable is non-negative.}
\[
0\leq Z_k(m)\leq Y_k(m)\leq X_k(m).
\]
We deduce from this and the triangle inequality that
\[
\left\| \W_n^{(\mathbf{M})}(X)-\W_n^{(\mathbf{M})}(Z)\right\|_\infty\leq n^{-1/\alpha}\sum_{k=\frac{1}{2\alpha}\log(n)+1}^{\frac{1}{\alpha}\log(n)}\left(S_n(X_k)-S_n(Z_k)\right). 
\]
We will show that the right hand side converges to $0$ in probability. 

Firstly $0<\alpha<1$, hence for all $k>\frac{1}{2\alpha}\log(n)$, $n<4^k$. Consequently by Fact \ref{fact:1},
\[
0\leq S_n(Y_k)-S_n(Z_k)\leq n4^{-k}\leq 1. 
\]
We conclude that 
\begin{equation}
I_n:=n^{-1/\alpha}\sum_{k=\frac{1}{2\alpha}\log(n)+1}^{\frac{1}{\alpha}\log(n)}\left(S_n(Y_k)-S_n(Z_k)\right)\lesssim \frac{\log(n)}{n^{1/\alpha}}. 
\end{equation}
Secondly, 
\[
n^{-1/\alpha}\sum_{k=\frac{1}{2\alpha}\log(n)+1}^{\frac{1}{\alpha}\log(n)}\left(S_n(X_k)-S_n(Y_k)\right)=\mathrm{II}_n+\mathrm{III}_n
\]
where 
\[
\widetilde{V}_k(m):=X_k(m)1_{\left[X_k(m)>4^k\right]}\ \ \text{and}\ \ \  \widehat{V}_k(m):=X_k(m)1_{\left[X_k(m)\leq 2^k\right]}, 
\]
and
\begin{align*}
\mathrm{II}_n&:=n^{-1/\alpha}\sum_{k=\frac{1}{2\alpha}\log(n)+1}^{\frac{1}{\alpha}\log(n)}S_n(\widehat{V}_k)\\
\mathrm{III}_n&:=n^{-1/\alpha}\sum_{k=\frac{1}{2\alpha}\log(n)+1}^{\frac{1}{\alpha}\log(n)}S_n(\widetilde{V}_k).
\end{align*}
Similarly to the proof of Lemma \ref{lem: a<1 Large}, 
\begin{align*}
\P\left(\mathrm{III}_n\neq 0\right)&\leq \sum_{k=\frac{1}{2\alpha}\log(n)+1}^{\frac{1}{\alpha}\log(n)}\sum_{j=1}^n\P\left(X_k(j)>4^k\right)\\
&\lesssim Cn\sum_{k=\frac{1}{2\alpha}\log(n)+1}^{\frac{1}{\alpha}\log(n)}\frac{4^{-k\alpha}}{k}\lesssim \frac{1}{\log(n)},
\end{align*}
showing that $\mathrm{III}_n\xrightarrow[n\to\infty]{}0$ in probability. 

We now fix $\alpha<r<1$ and $\varepsilon>0$. Note that by Corollary \ref{cor: moments} there exists a global constant $C$ so that for every $k$ and $m$,
\[
\E\left(\left(\widehat{V}_k(m)\right)^r\right)\leq C\frac{2^{k(r-\alpha)}}{k}. 
\]
By Markov's inequality and the triangle inequality for the $r$'th moments, 
\begin{align*}
\P\left(\mathrm{II}_n>\varepsilon\right)&\leq \E\left(\left(\mathrm{II}_n\right)^r\right)\varepsilon^{-r}\\
&\leq n^{-r/\alpha}\varepsilon^{-r}\sum_{k=\frac{1}{2\alpha}\log(n)+1}^{\frac{1}{\alpha}\log(n)}\sum_{j=1}^n\E\left(\left(\widehat{V}_k(j)\right)^r\right)\\
&\lesssim \varepsilon^{-r}n^{1-r/\alpha}\sum_{k=\frac{1}{2\alpha}\log(n)+1}^{\frac{1}{\alpha}\log(n)}\frac{2^{k(r-\alpha)}}{k}\lesssim \varepsilon^{-r}\frac{1}{\log(n)}. 
\end{align*}
We conclude that $\mathrm{II}_n\xrightarrow[n\to\infty]{}0$ in probability. 
Finally we conclude the proof as we have 
\[
\left\| \W_n^{(\mathbf{M})}(X)-\W_n^{(\mathbf{M})}(Z)\right\|_\infty\leq \mathrm{I}_n+\mathrm{II}_n+\mathrm{III}_n
\]
and each of the terms on the right hand side converge to $0$ in probability. 
\end{proof}

\begin{proof}[Proof of Lemma \ref{lem: a<1 middle Z}.(b)]
For all $0\leq t\leq 1$ 
\[
\W_n^{(\mathbf{M})}(X)(t)=n^{-1/\alpha}S_{[nt]}\big(V_n\big),
\]
where for $j\in\N$,
\[
V_n(j):=\sum_{k=\frac{1}{2\alpha}\log(n)+1}^{\frac{1}{\alpha}\log(n)}X_k(j)
\]
We claim that $V_n(1),V_n(2),\ldots,V_n(n)$ are i.i.d. $S_\alpha(A_n,1,0)$ random variables with $\lim_{n\to\infty}(A_n)^\alpha=\ln(2)$. 

Indeed, since $\alpha<1$, we deduce that for all $k>\frac{1}{2\alpha}\log(n)$, we have $4^{k}>n$. The independence of $V_n(1),V_n(2),\ldots,V_n(n)$ readily follows from the independence of $\left\{X_k(m):k\in\N,\ m\leq 4^{k}\right\}$.

 Now for all $1\leq j\leq n$ and $k\in\left(\frac{1}{2\alpha}\log(n),\frac{1}{\alpha}\log(n)\right]$, $X_k(j)$ is an $S_\alpha(\sigma_k,1,0)$ random variable with $(\sigma_k)^\alpha=\frac{1}{k}$. 
 As $V_n(j)$ is a sum of independent $S_\alpha(\sigma_k,1,0)$ random variables (and $\alpha\neq 1$), it follows from \cite[Properties 1.2.1 and 1.2.3]{MR1280932} that $V_n(j)$ is $S_\alpha(A_n,1,0)$ distributed with 
 \[
(A_n)^\alpha=\sum_{k=\frac{1}{2\alpha}\log(n)}^{\frac{1}{\alpha}\log(n)}\frac{1}{k}\sim \ln(2),\ \ \text{as}\ n\to\infty. 
 \]
We will now conclude the proof. Write $a_n:=\frac{(\ln(2))^{1/\alpha}}{A_n}$ and define $W_n(t):=a_n\W_n^{(\mathbf{M})}(X)(t)$ so that $W_n$ is the partial sum process driven by the random variables, $a_nV_n(1),\ldots, a_nV_n(n)$. 

As the latter are i.i.d. $S_\alpha(\ln(2),\beta,0)$ random variables, this shows that $W_n$ is equally distributed as $
\W_n(V)$ where $\left(V(j)\right)_{j=1}^\infty$ are i.i.d. $S_\alpha(\ln(2),1,0)$ random variables. 

By \cite[Corollary 7.1.]{MR2271424} $\W_n(V)$ (and hence $W_n$) converges in distribution to an $S_\alpha(\ln(2),1,0)$ L\'evy motion. 
	
Since $\W_n^{(\mathbf{M})}(X)=(a_n)^{-1}W_n$ with $a_n\to1$, we conclude that $\W_n^{(\mathbf{M})}(X)$ converges in distribution to an $S_\alpha\left(\sqrt[\alpha]{\ln(2)},1,0\right)$ L\'evy motion.
\end{proof}
\subsection{Concluding the proof of Theorem \ref{thm: main a<1}}\label{sub: Theorem general from skewed}
We now fix $\alpha\in (0,1)$ and $\beta\in[-1,1]$ and set $h_k$, $h$ be the functions from Theorem \ref{thm: main a<1} corresponding to $\beta$. We claim that 
$
\W_n(h)
$ converges in distribution to an $S_\alpha(\ln(2),\beta,0)$ L\'evy motion. 

We deduce this claim from the results on the skewed $\beta=1$ case via the following lemma. 
\begin{lemma}\label{lem: beta general 1}
$ $
\begin{itemize}
	\item[(a)] $\left(\W_n^{(\mathbf{S})}(f),\W_n^{(\mathbf{S})}(g)\right)$ converges in distribution (in the uniform topology) to $(0,0)$. 
	\item[(b)] $\left(\W_n^{(\mathbf{M})}(f),\W_n^{(\mathbf{M})}(g)\right)$ converges in distribution to $(\W,\W')$ where $\W,\W'$ are independent $S_\alpha\left(\sqrt[\alpha]{\ln(2)},1,0\right)$ L\'evy motions. 
	\item[(c)]$\left(\W_n^{(\mathbf{L})}(f),\W_n^{(\mathbf{L})}(g)\right)$ converges in distribution (in the uniform topology) to $(0,0)$. 
\end{itemize}
\end{lemma}
\begin{proof}
For all $k\in\mathbb{N}$, $f_k$ and $g_k$ are equally distributed. Following the proofs of Lemmas \ref{lem: a<1 Large} and \ref{lem: a<1 small} we see that
$
\left\|W_n^{(\mathbf{S})}(g)\right\|_\infty$ and  $\left\|W_n^{(\mathbf{L})}(g)\right\|_\infty$ tend to $0$ in probability as $n\to\infty$. Parts (a) and (c) follow from this and Lemmas \ref{lem: a<1 Large} and \ref{lem: a<1 small}.

Now for all $n\in\N$, $\W_n^{(\mathbf{M})}(f)$ and $\W_n^{(\mathbf{M})}(g)$ are independent and equally distributed. Part (b) follows from this and Proposition \ref{prop: a<1 middle h}. 
\end{proof}
An immediate corollary is the following. 
\begin{corollary}\label{cor: beta general}
	$ $
	\begin{itemize}
		\item[(a)] $\left\|\W_n^{(\mathbf{S})}(h)\right\|_\infty\to 0$ converges in measure. 
		\item[(b)] $\W_n^{(\mathbf{M})}(h)$ converges in distribution to an $S_\alpha\left(\sqrt[\alpha]{\ln(2)},\beta,0\right)$ L\'evy motion. 
		\vspace{0.45mm}
		\item[(c)]$\left\|\W_n^{(\mathbf{L})}(h)\right\|_\infty \to0$ in measure. 
	\end{itemize}
\end{corollary}
\begin{proof}
Set 
\[
\varphi(x,y)=\left(\frac{1+\beta}{2}\right)^{1/\alpha}x-\left(\frac{1-\beta}{2}\right)^{1/\alpha}y
\]
and write $c_\beta:=\left(\left(\frac{\beta+1}{2}\right)^{1/\alpha}-\left(\frac{1-\beta}{2}\right)^{1/\alpha}\right)$.

 For each $D\in\{\mathbf{S},\mathbf{M},\mathbf{L}\}$, and all $n\in\N$, 
\[
\varphi\left(\W_n^{(\mathbf{D})}(f),\W_n^{(\mathbf{D})}(g)\right)=\W_n^{(\mathbf{D})}(h)
\]

Parts (a) and (c) follow from Lemma \ref{lem: beta general 1}.(a) and (c) since for all $x,y\in\R$,  $|\varphi(x,y)|\leq |x|+|y|$. 

Let $\W,\W'$ be two independent $S_\alpha\left(\sqrt[\alpha]{\ln(2)},1,0\right)$ L\'evy motions. It follows that  $\widetilde{\W}:=\varphi(\W,\W')$ is a process with independent increments. By \cite[Property 1.2.13]{MR1280932}, for all $s<t$, $\widetilde{\W}(t)-\widetilde{\W}(s)$ is $S_\alpha(\sqrt[\alpha]{\ln(2)(t-s)},1,0)$ distributed, whence $\widetilde{\W}$ is a $S_\alpha\left(\sqrt[\alpha]{\ln(2)},\beta,0\right)$ L\'evy motion. 

Since $\varphi$ is continuous, Lemma \ref{lem: beta general 1} and the continuous mapping theorem imply that 
\[
\varphi\left(\W_n^{(\mathbf{M})}(f),\W_n^{(\mathbf{M})}(g)\right)=\W_n^{(\mathbf{M})}(h)
\]
 converges in distribution to $\widetilde{\W}$ and the proof is concluded. 
\end{proof}
\begin{proof}[Proof of Theorem \ref{thm: main a<1}]
	By Corollary \ref{cor: beta general}.(a) and (c), $\W_n^{(\mathbf{S})}(h)+\W_n^{(\mathbf{L})}(h)$ converge in distribution to the $0$ function. The theorem then follows from \eqref{eq: decomp}, Corollary \ref{cor: beta general}.(b) and Lemma \ref{lem: conv together D}. 
\end{proof}

\section{Proof of Theorem \ref{thm: a>1}}
Let $\alpha\geq 1$.  The strategy of the proof goes along similar lines. However there is a major difference in the treatment of $\W_n^{(\mathbf{S})}$ as the $L^1$ norm does not decay to $0$. For this reason  we retort to a more sophisticated $L^2$ estimate and make use of the fact that for all $k$, $h_k$ is a $T^{4^{k^2}}$ co-boundary. 

In what follows $1\leq \alpha<2$ is fixed, $h_k$ and $h$ are as in the statement of Theorem \ref{thm: a>1} and the decomposition of $\W_n(h)$ to a sum of $\W_n^{(\mathbf{S})}(h)$, $\W_n^{(\mathbf{S})}(h)$ and $\W_n^{(\mathbf{S})}(h)$ is as before. We write $d_k:=4^{k^2}$. 

\begin{lemma}\label{lem: a>1 Large}
	$\lim_{n\to\infty}m\left(\left\|\W_n^{(\mathbf{L})}(h)\right\|_\infty \neq 0\right)=0$. 
\end{lemma}
\begin{proof}
We have the inclusion 
\[
\left[\left\|\W_n^{(\mathbf{L})}(h)\right\|_\infty \neq 0\right]\subset \bigcup_{k=\frac{1}{\alpha}\log(n)+1}^\infty \bigcup_{j=0}^{n-1}\left[f_k\circ T^j\neq 0\ \ \text{or}\ \ f_k\circ T^{d_k+j}\neq 0\right].
\]
In a similar way to the proof of Lemma \ref{lem: a<1 Large}, we have 
\[
\P\left(\left\|\W_n^{(\mathbf{L})}(h)\right\|_\infty \neq 0\right)\leq \sum_{k=\frac{1}{\alpha}\log(n)+1}^\infty 2n \cdot m\left(f_k\neq 0\right)\lesssim \frac{1}{\log(n)}\xrightarrow[n\to\infty]{}0. 
\]
\end{proof}
As before we also have. 
\begin{lemma}\label{lem: a>1 small sym}
	$\left\|\W_n^{(\mathbf{S})}(f)\right\|_\infty\xrightarrow[n\to\infty]{}0$ in measure. 
\end{lemma}
The proof of this lemma when $1\leq \alpha<2$ is more difficult than the analogous Lemma \ref{lem: a<1 small}. It is given in Subsection \ref{sub: a>1 symmeteric small}. 

\begin{proposition}\label{prop: a>1 middle h}
	$\W_n^{(\mathbf{M})}(h)$ converges in distribution to $\W$, a $\aS\left(\sqrt[\alpha]{\ln(2)}\right)$ L\'evy motion. 
\end{proposition}
Assuming the previous claims we can complete the proof of Theorem \ref{thm: a>1}.
\begin{proof}[Proof of Theorem \ref{thm: a>1}]
By Lemmas \ref{lem: a>1 Large} and \ref{lem: a>1 small sym}, $\left\|\W_n^{(\mathbf{S})}+\W_n^{(\mathbf{L})}\right\|_\infty$ converges in probability to $0$. The claim now follows from Proposition \ref{prop: a>1 middle h} and Lemma \ref{lem: conv together D}. 
\end{proof}	
In the next two subsections we prove Proposition \ref{prop: a>1 middle h} and Lemma \ref{lem: a>1 small sym} respectively. 
\subsection{Proof of Proposition \ref{prop: a>1 middle h}}

We introduce the following $D[0,1]$ valued processes on $(\Omega,\mathcal{F},\P)$:
\begin{align*}
	\W_n^{(\mathbf{M})}(Z)(t)&:=\frac{1}{n^{1/\alpha}}\sum_{k=\frac{1}{2\alpha}\log(n)+1}^{\frac{1}{\alpha}\log(n)}S_{[nt]}\left(Z_k(\cdot)-Z_k\left(\cdot+d_k\right)\right),\\ \W_n^{(\mathbf{M})}(Y)(t)&:=\frac{1}{n^{1/\alpha}}\sum_{k=\frac{1}{2\alpha}\log(n)+1}^{\frac{1}{\alpha}\log(n)}S_{[nt]}\left(Y_k(\cdot)-Y_k\left(\cdot+d_k\right)\right)\ \ \text{and}\\
	\W_n^{(\mathbf{M})}(X)(t)&:=\frac{1}{n^{1/\alpha}}\sum_{k=\frac{1}{2\alpha}\log(n)+1}^{\frac{1}{\alpha}\log(n)}S_{[nt]}\left(X_k(\cdot)-X_k\left(\cdot+d_k\right)\right).
\end{align*}	
The following is the analogue of Lemma \ref{lem: a<1 equal dist} for the current case. 
\begin{lemma} \label{lem: a>1 equal dist}
	$\W_n^{(\mathbf{M})}(h)$ and $\W_n^{(\mathbf{M})}(Z)$ are equally distributed. 
\end{lemma}
The proof of Lemma \ref{lem: a>1 equal dist} is similar to the proof of Lemma \ref{lem: a<1 equal dist} with obvious modifications. We leave it to the reader. Proposition \ref{prop: a>1 middle h} follows from Lemma \ref{lem: a>1 equal dist} and the following. 

\begin{lemma}\label{lem: a>1 middle Z}
	$ $
	\begin{itemize}
		\item[(a)] $
		\left\| \W_n^{(\mathbf{M})}(X)-\left(\W_n^{(\mathbf{M})}(Z)\right)\right\|_\infty\xrightarrow[n\to\infty]{}0$ in measure. 
		\item[(b)] $\W_n^{(\mathbf{M})}(X)$ converges in distribution to a $\aS\left(\sqrt[\alpha]{2\ln(2)}\right)$ L\'evy motion.
	\end{itemize}
	Consequently  $\W_n^{(\mathbf{M})}(Z)$ converges in distribution to a $\aS\left(\sqrt[\alpha]{2\ln(2)}\right)$ L\'evy motion
\end{lemma}
\begin{proof}[Proof of Lemma \ref{lem: a>1 middle Z}.(b)]
	For all $0\leq t\leq 1$ 
	\[
	\W_n^{(\mathbf{M})}(X)(t)=n^{-1/\alpha}S_{[nt]}\big(V_n\big),
	\]
	where for $j\in\N$,
	\[
	V_n(j):=\sum_{k=\frac{1}{2\alpha}\log(n)+1}^{\frac{1}{\alpha}\log(n)}\left(X_k(j)-X_k\big(j+d_k\big)\right).
	\]
	We claim that for all but finitely many $n$, $V_n(1),V_n(2),\ldots,V_n(n)$ are i.i.d. $\aS(A_n)$ random variables with $\lim_{n\to\infty}(A_n)^\alpha=\ln(2)$. 
	
	For all $n\geq 2^{4\alpha}$, if $k\geq \frac{1}{2\alpha}\log(n)$, we have $d_k\geq n$. For all such $n$, the independence of $V_n(1),\ldots,V_n(n)$ follows from the independence of $\left\{X_k(j):\ k\in\N,1\leq j\leq 2\cdot d_k\right\}$. We will now calculate its distributions.  
	
	For all $1\leq j\leq n$ and $k>\frac{1}{2\alpha}\log(n)$, $X_k(j)-X_k\big(j+d_k\big)$ is a difference of two independent $S_\alpha\left(k^{-1/\alpha},1,0\right)$ random variables. By \cite[Properties 1.2.1 and 1.2.3]{MR1280932}, it is $\aS\left(\left(\frac{2}{k}\right)^{1/\alpha}\right)$ distributed. As $V_n(j)$ is a sum of independent $\aS$ random variables, we see that $V_n(j)$ is $\aS(A_n)$ distributed with
	\[
	(A_n)^\alpha:=\sum_{k=\frac{1}{2\alpha}\log(n)}^{\frac{1}{\alpha}\log(n)}\frac{2}{k}=2\ln(2)(1+o(1)),\ \ \text{as}\ n\to\infty. 
	\]
	This concludes the claim on $V_n(1),\ldots,V_n(n)$. The conclusion of the statement from here is similar to the end of the proof of Lemma \ref{lem: a>1 middle Z}.(b).
\end{proof}
\begin{proof}[Proof of Lemma \ref{lem: a>1 middle Z}.(a)]
We assume $n>2^{4\alpha}$ so that for all $k>\frac{1}{2\alpha}\log(n)$, $d_k>n$. 

Firstly, since for all $k\in\N$ and $j\leq 2d_k$, 
\[
0<Y_k(j)-Z_k(j)\leq 4^{-{k}},
\]
we have
\begin{align*}
	n^{1/\alpha}\left\| \W_n^{(\mathbf{M})}(Y)-\left(\W_n^{(\mathbf{M})}(Z)\right)\right\|_\infty&\leq \sum_{k=\frac{1}{2\alpha}\log(n)+1}^{\frac{1}{\alpha}\log(n)}S_n\left(Y_k(\cdot)-Z_k\left(\cdot\right)\right)\\
	&\ \ \  +\sum_{k=\frac{1}{2\alpha}\log(n)+1}^{\frac{1}{\alpha}\log(n)}S_n\left(Y_k(\cdot+d_k)-Z_k\left(\cdot+d_k\right)\right)\\
	&\leq 2n\sum_{k=\frac{1}{2\alpha}\log(n)+1}^{\frac{1}{\alpha}\log(n)}4^{-k}\lesssim n^{1-\frac{1}{\alpha}}.
\end{align*}
Consequently,
\begin{equation}\label{eq: a>1 W_n(Y)-W_n(Z)}
\left\| \W_n^{(\mathbf{M})}(Y)-\W_n^{(\mathbf{M})}(Z)\right\|_\infty\lesssim n^{1-\frac{2}{\alpha}}\xrightarrow[n\to\infty]{}0.
\end{equation}
We turn to look at $\W_n^{(\mathbf{M})}(X)-\W_n^{(\mathbf{M})}(Y)$. For all $0\leq t \leq 1$, 
\[
\W_n^{(\mathbf{M})}(X)(t)-\W_n^{(\mathbf{M})}(Y)(t)=\mathrm{II}_n(t)+\mathrm{III}_n(t),
\]
where 
\begin{align*}
\widetilde{V}_k(m)&:=X_k(m)1_{\left[X_k(m)>4^k\right]}-X_k\big(m+d_k\big)1_{\left[X_k\big(m+d_k\big)> 4^k\right]}\ \ \ \ \text{and}\\ 
 \widehat{V}(m)&:=\sum_{k=\frac{1}{2\alpha}\log(n)+1}^{\frac{1}{\alpha}\log(n)}\left(X_k(m)1_{\left[X_k(m)\leq 2^k\right]}-X_k\big(m+d_k\big)1_{\left[X_k\big(m+d_k\big)\leq 2^k\right]}\right), 
\end{align*}
and
\begin{align*}
	\mathrm{II}_n(t)&:=n^{-1/\alpha}S_{[nt]}\left(\widehat{V}\right)\\
	\mathrm{III}_n(t)&:=n^{-1/\alpha}\sum_{k=\frac{1}{2\alpha}\log(n)+1}^{\frac{1}{\alpha}\log(n)}S_{[nt]}(\widetilde{V}_k).
\end{align*}
Similarly to the proof of Lemma \ref{lem: a<1 Large}, 
\begin{align*}
	\P\left(\exists t: \mathrm{III}_n(t)\neq 0\right)&\leq \sum_{k=\frac{1}{2\alpha}\log(n)+1}^{\frac{1}{\alpha}\log(n)}\sum_{j=1}^n\P\left(\widetilde{V}_k(j)\neq 0\right)\\
	&\leq \sum_{k=\frac{1}{2\alpha}\log(n)+1}^{\frac{1}{\alpha}\log(n)}\sum_{j=1}^n\left(\P\left(X_k(m)>4^k\right)+\P\left(X_k\big(m+d_k\big)>4^k\right)\right)\\
	&\lesssim 2Cn\sum_{k=\frac{1}{2\alpha}\log(n)+1}^{\frac{1}{\alpha}\log(n)}\frac{4^{-k\alpha}}{k}\lesssim \frac{1}{\log(n)},
\end{align*}
Now $ \widehat{V}(j)$, $1\leq j \leq n$, are zero-mean, independent random variables. By Proposition \ref{prop: second moment for a>=1}, they also have second moment and for all $1\leq j\leq n$, 
\begin{align}
\E\left(\widehat{V}(j)^2\right)&\lesssim 2\sum_{k=\frac{1}{2\alpha}\log(n)+1}^{\frac{1}{\alpha}\log(n)}\E\left(X_k(j)^21_{\left[X_k(j)\leq 2^k\right]}\right) \nonumber \\
&\lesssim  \sum_{k=\frac{1}{2\alpha}\log(n)+1}^{\frac{1}{\alpha}\log(n)}\frac{2^{(2-\alpha)k}}{k}\lesssim \frac{n^{\frac{2}{\alpha}-1}}{\log(n)}.
\end{align}
It follows from Kolmogorov's maximal inequality that for every $\epsilon>0$,
\begin{align*}
\P\left(\max_{0\leq t\leq 1}\big|\mathrm{II}_n(t)\big|>\epsilon\right)&= \P\left(\max_{1\leq m\leq n}\left|S_m\left(\widehat{V}\right)\right|>\epsilon n^{1/\alpha}\right)\\
&\leq \epsilon^{-2\alpha}n^{-2/\alpha}\E\left(\left|S_n\left(\widehat{V}\right)\right|^2\right)\\
&= \epsilon^{-2\alpha}n^{-2/\alpha}\left(n\E\left(\widehat{V}(1)^2\right)\right)\lesssim \frac{1}{\epsilon^{2\alpha}\log(n)}.
\end{align*}
Here the first equality of the last line is true as $\widehat{V}(1),\ldots,\widehat{V}(n)$ are independent, zero-mean random variables with finite variance. This concludes the proof that 
\[
\left\|\mathrm{II}_n\right\|_\infty\xrightarrow[n\to\infty]{}0\ \ \text{in probability}. 
\]
The claim now follows from \eqref{eq: a>1 W_n(Y)-W_n(Z)} and the convergence in probability of $\mathrm{II}_n,\mathrm{III}_n$ to the zero function. 
\end{proof}	

\subsection{Proof of Lemma \ref{lem: a>1 small sym}}\label{sub: a>1 symmeteric small}
We first write 
\[
\W_n^{\mathbf{S}}(h)=\W_n^{(\mathbf{VS})}(h)+\W_n^{(\mathbf{LS})}(h)
\]
where 
\begin{align*}
\W_n^{(\mathbf{VS})}(h)&=\sum_{k=1}^{\sqrt{\log(n)}} \W_n(h_k)\\
\W_n^{(\mathbf{LS})}(h)&=\sum_{k=\sqrt{\log(n)}+1}^{\frac{1}{2\alpha}\log(n)} \W_n(h_k).
\end{align*}
The reason for this further decomposition is that $d_k>n$ if and only if $k>\sqrt{\log(n)}$ so that only in the very small $(\mathbf{VS})$ terms we no longer have full independence in the summands. The proof that $\W_n^{(\mathbf{LS})}(h)$ tends to the zero function is quite similar to the proof of the last part in Lemma \ref{lem: a>1 middle Z}.(a) while the proof of the other term makes use of the fact that we are dealing with coboundaries. 
\begin{lemma}\label{lem: a>1 LS}
$\left\|\W_n^{(\mathbf{LS})}(h)\right\|_\infty\xrightarrow[n\to\infty]{}0$. 
\end{lemma}
\begin{proof}
Write 
\[
\psi_n:=\sum_{k=\sqrt{\log(n)}+1}^{\frac{1}{2\alpha}\log(n)}\left(f_k-f_k\circ T^{d_k}\right).
\]
We have 
\begin{itemize}
	\item $\psi_n\circ T^j$, $1\leq j\leq n$ are independent (since $d_k>n$ for all $k$ in the range of summation), bounded and $\int \psi_n dm=0$. 
	\item for all $t$, $\W_n^{(\mathbf{LS})}(h)(t)=n^{-1/\alpha}S_{[nt]}(\psi_n)$. 
\end{itemize} 
By Kolmogorov's maximal inequality, for all $\epsilon>0$
\begin{align*}
m\left(\left\|\W_n^{(\mathbf{LS})}(h)\right\|_\infty>\epsilon\right)&=m\left(\max_{1\leq k\leq n}\left|S_k\left(\psi_n\right)\right|>\epsilon n^{1/\alpha}\right)\\
&\leq \frac{\left\|S_n\left(\psi_n\right)\right\|_2^2}{\epsilon^{2}n^{2/\alpha}}=\epsilon^{-2}n^{1-\frac{2}{\alpha}}\|\psi_n\|_2^2,
\end{align*}
here the last equality follows from $S_n\left(\psi_n\right)$ being a sum of zero mean, square integrable, independent random variables. We will now give an upper bound for $\|\psi_n\|_2^2$.  Firstly, $\left\{f_k-f_k\circ T^{d_k}:\ k>\sqrt{\log(n)}\right\}$ is distributed as $\left\{Z_k(1)-Z_k(d_k+1):\ k>\sqrt{\log(n)}\right\}$.  Using in addition that for all $k\in \N$,
 \[
Z_k(1)\leq Y_k(1)\leq \left|X_k(1)\right|1_{\left[X_k(1)\leq 4^k\right]},
\] 
we observe that
\begin{align*}
\|\psi_n\|_2^2&= \E\left(\left(\sum_{k=\sqrt{\log(n)}+1}^{\frac{1}{2\alpha}\log(n)}\left(Z_k(1)-Z_k(1+d_k)\right)\right)^2\right)\\
&=\sum_{k=\sqrt{\log(n)}+1}^{\frac{1}{2\alpha}\log(n)}\E\left(\left(Z_k(1)-Z_k(1+d_k)\right)^2\right)\\
&\leq 4\sum_{k=\sqrt{\log(n)}+1}^{\frac{1}{2\alpha}\log(n)}\E\left(\left(Z_k(1)\right)^2\right)\\
&\leq 4\sum_{k=\sqrt{\log(n)}+1}^{\frac{1}{2\alpha}\log(n)}\E\left(\left|X_k(1)\right|^21_{\left[X_k(1)\leq 4^k\right]}\right),\ \ \text{by Proposition \ref{prop: second moment for a>=1}}\\
&\lesssim 4\sum_{k=\sqrt{\log(n)}+1}^{\frac{1}{2\alpha}\log(n)}C\frac{4^{(2-\alpha)k}}{k}\lesssim \frac{n^{\frac{2}{\alpha}-1}}{\log(n)}. 
\end{align*} 
Plugging this in the previous upper bound, we see that for all $\epsilon>0$, 
\[
m\left(\left\|\W_n^{(\mathbf{LS})}(h)\right\|_\infty>\epsilon\right)\lesssim \epsilon^{-2}n^{1-\frac{2}{\alpha}}\|\psi_n\|_2^2\lesssim \frac{\epsilon^{-2}}{\log(n)}\xrightarrow[n\to\infty]{}0,
\]
proving the claim. 
\end{proof}

We turn to treat $\W_n^{(\mathbf{VS})}(h)$. As before we define
\[
\varphi_n:=\sum_{k=1}^{\sqrt{\log(n)}}\left(f_k-f_k\circ T^{d_k}\right),
\]
so that for all $t\in[0,1]$, $\W_n^{(\mathbf{VS})}(h)=S_{[nt]}\big(\varphi_n\big)$. It is no longer guaranteed that $\varphi_n,\ldots,\varphi_n\circ T^n$ are independent. For this reason we can no longer bound the maximum using L\'evy inequality and we will make use of a more general maximal inequality. The first step involves bounding the square moments of random variables and we make repetitive use of the following most crude bound; 

\begin{claim}\label{claim: crude variance bound}
Let $U_1,U_2,\ldots,U_N$ be square integrable random variables, then 
\[
\E\left(\left(\sum_{j=1}^NU_j\right)^2\right)\leq N\sum_{j=1}^N \E\left(\big(U_j\big)^2\right)
\]	
\end{claim}



\begin{lemma}\label{lem: a>1 such a long one}
There exists a global constant $C>0$ so that for all $1\leq l<j\leq n$ and $1\leq k\leq \sqrt{\log(n)}$, 
\[
\left\|S_j\left(f_k-f_k\circ T^{d_k}\right)-S_l\left(f_k-f_k\circ T^{d_k}\right)\right\|_2^2\le C(j-l)\frac{4^{(2-\alpha)k}}{k}.
\]
\end{lemma}
\begin{proof}
Let $\mu_k:=\int f_kdm$ and write $F_k:=f_k-\mu_k$. For every $j\leq n$, 
\begin{align*}
S_j\left(f_k-f_k\circ T^{d_k}\right)&=S_j\left(F_k-F_k\circ T^{d_k}\right)\\
&=S_{\min(j,d_k)}(F_k)-S_{\min(j,d_k)}\big(F_k\big)\circ T^{\max(j,d_k)}.
\end{align*}
Consequently, for every $1\leq l<j\leq n$, 
\[
S_j\left(f_k-f_k\circ T^{d_k}\right)-S_l\left(f_k-f_k\circ T^{d_k}\right)=A+B,
\]
where
\[
A:=\begin{cases}
	\vspace{1.8mm}	\sum_{r=l}^{j-1}F_k\circ T^r,& l<j<d_k,\\
	\vspace{1.8mm}	\sum_{r=l}^{d_k-1}F_k\circ T^r,&l<d_k\leq j,\\
	0, &\text{otherwise}. 
\end{cases}
\]
and 
\[
B:=\begin{cases}
\vspace{1.8mm}	-\sum_{r=d_k+l}^{d_k+j-1}F_k\circ T^{r},& l<j\leq d_k,\\
\vspace{1.8mm}	\sum_{r=d_k}^{j-1}F_k\circ T^r-\sum_{r=d_k+l}^{d_k+j-1}F_k\circ T^r,& l< d_k<j\leq l+d_k, \\
\vspace{1.8mm} \sum_{r=l}^{j-1}F_k\circ T^r-\sum_{r=d_k+l}^{d_k+j-1}F_k\circ T^r,&j-l\leq d_k\leq l<j,\\	
\sum_{r=\max(j,d_k)}^{j+d_k-1}F_k\circ T^r-\sum_{r=l}^{l+d_k-1}F_k\circ T^r,& j-l>d_k. 
\end{cases}
\]
We will next show that there exists a constant $C$ so that
 \begin{equation}
 \|A\|_2^2,\|B\|_2^2\leq 4C(j-l)\frac{4^{(2-\alpha)k}}{k},
 \end{equation}
The statement follows from  this and Claim \ref{claim: crude variance bound}. 

Recall that for all $0\leq L\leq d_k$ and $M\in\N$, $$S_L(F_k)\circ T^M=\sum_{r=M}^{M+L-1}F_k\circ T^r$$ is a sum of i.i.d. zero mean square integrable random variables. We deduce that so long as $L\leq d_k$, 
\begin{align*}
\left\|S_L(F_k)\circ T^M\right\|_2^2=L\left\|F_k\right\|_2^2. 
\end{align*}
A similar argument as in the proof of Lemma \ref{lem: a>1 LS} shows that there exists $c>0$ so that 
\[
 \|F_k\|_2^2\propto \|f_k\|_2\leq c\frac{4^{(2-\alpha)k}}{k} . 
\]
We conclude that there exists $c_2>0$ so that for all $L\leq d_k$ and $M\in\N$
\begin{equation}\label{eq: core bound VS}
\left\|S_L(F_k)\circ T^M\right\|_2^2\leq c_2 L\frac{4^{(2-\alpha)k}}{k}.
\end{equation}
Noting that in the definition of $A$ all terms on the right are of the form $S_L(F_k)\circ T^M$ with $L\leq d_k$, we observe that 
\[
\|A\|_2^2\leq c_2\frac{4^{(2-\alpha)k}}{k}\begin{cases}
j-l.& l<j\leq d_k,\\
d_k-j,&\ l\leq d_k<j,\\
0, &\text{otherwise},
\end{cases}
\]
and thus
\[
\|A\|_2^2\leq c_2(j-l)\frac{4^{(2-\alpha)k}}{k}.
\]
Now by Claim \ref{claim: crude variance bound},
\[
\|B\|_2^2\leq \begin{cases}
	\left\|S_{j-l}(F_k)\circ T^{d_k}\right\|_2^2,& l<j\leq d_k,\\
	2\left(\left\|S_{j-d_k}(F_k)\circ T^{d_k}\right\|_2^2+\left\|S_{j-l}(F_k)\circ T^{l+d_k}\right\|_2^2\right),& l\leq d_k<j\leq l+d_k, \\
	2\left(\left\|S_{j-l}(F_k)\circ T^{l}\right\|_2^2+\left\|S_{j-l}(F_k)\circ T^{d_k+l}\right\|_2^2\right),&j-l\leq d_k\leq l<j,\\
	2\left(\left\|S_{\min(l,d_k)}(F_k)\circ T^{\max(l,d_k)}\right\|_2^2+\left\|S_{d_k}(F_k)\circ T^l\right\|_2^2\right),& j-l>d_k. 
\end{cases}
\]
A similar argument as the one on $\|A\|_2^2$ shows that 
\[
\|B\|_2^2\leq c_2\frac{4^{(2-\alpha)k}}{k}\begin{cases}
	(j-l),& l<j\leq d_k,\\
	2\left((j-d_k)+(j-l)\right),& l\leq d_k<j\leq l+d_k, \\
	4(j-l)& j-l\leq d_k\leq l<j,\\
	2\left(\min(l,d_k)+d_k\right),& j-l>d_k,
\end{cases}
\]
and
\[
\|B\|_2^2\leq 4c_2(j-l)\frac{4^{(2-\alpha)k}}{k}.
\]
This concludes the proof. 
\end{proof}
\begin{corollary}\label{cor: moment VS}
For every $\kappa>0$, there exists $C>0$ such that for all $1\leq l<j\leq n$, 
\[
\left\|S_j\big(\varphi_n\big)-S_l\big(\varphi_n\big)\right\|_2^2\leq C(j-l)n^\kappa. 
\]
\end{corollary}
\begin{proof}
By Claim \ref{claim: crude variance bound}, 
\[
\left\|S_j\big(\varphi_n\big)-S_l\big(\varphi_n\big)\right\|_2^2\leq \sqrt{\log(n)}\sum_{k=1}^{\sqrt{\log(n)}}\left\|S_j\left(f_k-f_k\circ T^{d_k}\right)-S_l\left(f_k-f_k\circ T^{d_k}\right)\right\|_2^2.
\]
Plugging in the bound of Lemma \ref{lem: a>1 such a long one} on the right hand side we see that there exists $C>0$ such that 
\[
\left\|S_j\big(\varphi_n\big)-S_l\big(\varphi_n\big)\right\|_2^2\leq C(j-l) \sqrt{\log(n)}\sum_{k=1}^{\sqrt{\log(n)}}\frac{4^{(2-\alpha)k}}{k}.
\]
Since 
\[
\sum_{k=1}^{\sqrt{\log(n)}}\frac{4^{(2-\alpha)k}}{k}\lesssim \frac{4^{2\sqrt{\log(n)}}}{\sqrt{\log(n)}} \ll \frac{n^\kappa}{\sqrt{\log(n)}},
\]
the claim follows.
\end{proof}
\begin{lemma}\label{lem: a>1 VS}
	$\left\|\W_n^{(\mathbf{VS})}(h)\right\|_\infty\xrightarrow[n\to\infty]{}0$. 
\end{lemma}
\begin{proof}
	Let $\epsilon>0$. We have
	\[
	m\left(\left\|\W_n^{(\mathbf{VS})}(h)\right\|_\infty>\epsilon\right)=m\left(\max_{1\leq j\leq n}\left|S_j\big(\varphi_n\big)\right|>\epsilon n^{1/\alpha}\right).
	\]
	Fix $\kappa>0$ small enough so that $\kappa+1<\frac{2}{\alpha}$. By Corollary \ref{cor: moment VS} and Markov's inequality, for all $1\leq l<j\leq n$, 
	\[
	m\left(\left|S_j\big(\varphi_n\big)-S_l\big(\varphi_n\big)\right|>\epsilon n^{1/\alpha}\right)\leq C\epsilon^{-2}(j-l)n^{\kappa-\frac{2}{\alpha}}.
	\] 
	By \cite[Theorem 10.2]{MR1700749} with $\beta=\frac{1}{2}$ and $u_l:=\sqrt{C}n^{\frac{\kappa}{2}-\frac{1}{\alpha}}$, 
	\[
	m\left(\max_{1\leq j\leq n}\left|S_j\big(\varphi_n\big)\right|>\epsilon n^{1/\alpha}\right)\leq C\epsilon^{-2}n^{1+\kappa-\frac{2}{\alpha}}\xrightarrow[n\to\infty]{}0. 
	\]
\end{proof}
\begin{proof}[Proof of Lemma \ref{lem: a>1 small sym}]
The conclusion follows from Lemmas \ref{lem: a>1 LS} and \ref{lem: a>1 VS} and the triangle inequality. 		
\end{proof}	
\section{skewed CLT for $\alpha\in(1,2)$}\label{sec: Skewed CLT a>1}
Assume $\alpha\in(1,2)$ and $(f_k)_{k=1}^\infty$ are the functions from Corollary \ref{cor: weak Sinai} where $X_k(j)$ are $S_\alpha\left(\sqrt[\alpha]{1/l},1,0\right)$ random variables and $Z_k(j)$ is the corresponding discretisation of the truncation $Y_k(j)$.
Recall that  $D_k:=4^{\alpha k}$,
\[
\varphi_k:=\frac{1}{D_k}\sum_{j=0}^{D_k-1}f_k\circ T^j,
\]
$h_k:=f_k-\varphi_k$ and $h=\sum_{k=1}^\infty h_k$. The function $h$ is well defined by Lemma \ref{lem: h is well defined a>1}.

We aim to show that 
\[
\frac{S_n(h)+B_n}{n^{1/\alpha}}\Rightarrow^d S_\alpha(\ln(2),1,0),
\]
where 

\[
B_n:=n\sum_{k=\frac{1}{2\alpha}\log(n)}^{\frac{1}{\alpha}\log(n)}\E\left(X_k(1)1_{\left[X_k(1)\leq 2^k\right]}\right).
\]

\subsection{Proof of Theorem \ref{skewed CLT for a>1}}
The strategy of the proof starts with the decomposition,
\begin{equation}\label{eq: decomp for a>1}
S_n(h)+B_n=S_n^{(\mathbf{S})}(h)+S_n^{(\mathbf{M})}(f)+S_n^{(\mathbf{L})}(f)-V_n(\varphi)
\end{equation}
where
\begin{align*}
	S_n^{(\mathbf{S})}(h)&:=\sum_{k=1}^{\frac{1}{2\alpha}\log(n)}S_{n}\left(h_k\right),\\ S_n^{(\mathbf{M})}(f)&:=B_n+\sum_{k=\frac{1}{2\alpha}\log(n)+1}^{\frac{1}{\alpha}\log(n)}S_{n}\left(f_k-\int f_k dm\right)\\
	S_n^{(\mathbf{L})}(f)&:=\sum_{k=\frac{1}{\alpha}\log(n)+1}^{\infty}S_{n}\left(f_k-\int f_kdm\right) \ \ \text{and}\\
	V_n(\varphi)&:=\sum_{k=\frac{1}{2\alpha}\log(n)+1}^{\infty}S_{n}\left(\varphi_k-\int\varphi_kdm\right).
\end{align*}	
Note that in deriving \eqref{eq: decomp for a>1} we used that for all $k\in\N$, $\int f_kdm=\int \varphi_kdm$ and that both $\sum_{k=1}^N f_k$ and $\sum_{k=1}^N\varphi_k$ converge in $L^1(m)$ as $N\to\infty$. 

The proof of Theorem \ref{skewed CLT for a>1} is by showing that when normalized, three of the four terms converge to $0$ in probability and the remaining one converges in distribution to an $S_\alpha(\ln(2),0,0)$ random variable.

 \begin{lemma}\label{lem: a>>1 large}
 \[
\lim_{n\to\infty} m\left(S_n^{(\mathbf{L})}(f)\neq 0\right)=0.
 \]
\end{lemma} 
 The proof of Lemma \ref{lem: a>>1 large} is similar  to the proof of Lemma \ref{lem: a<1 Large} and is thus omitted. We next show.
 \begin{lemma}\label{lem: a>>1 phi}
$n^{-1/\alpha}V_n(\varphi)$ converges to $0$ in probability.  
 \end{lemma}
  
The proof of Lemma \ref{lem: a>>1 phi} begins with the following easy calculation.
\begin{fact}
If $n\leq D_k$ then
\begin{equation}\label{eq: sum phi n<d}
S_n\left(\varphi_k\right)=
\sum_{j=0}^{n-2}\frac{j+1}{D_k}f_k\circ T^j+\frac{n}{D_k}\sum_{j=n-1}^{D_k-1}f_k\circ T^j+\sum_{j=1}^{n-1}\frac{n-j}{D_k}f_k\circ T^{D_k+j-1}
\end{equation}
If $D_k\leq n$ then
\begin{equation}\label{eq: sum phi n>d}
	S_n\left(\varphi_k\right)=
	\sum_{j=0}^{D_k-2}\frac{j+1}{D_k}f_k\circ T^j+\sum_{j=D_k-1}^{n-1}f_k\circ T^j+\sum_{j=1}^{D_k-1}\frac{D_k-j}{D_k}f_k\circ T^{n+j-1}
\end{equation}
\end{fact}	
Since $f_k$ and $Z_k(1)$ are equally distributed and $Z_k(1)\leq Y_k(1)$, the next claim follows easily from Proposition \ref{prop: second moment for a>=1}.
\begin{claim}\label{clm: varf}
	For every $k\in\N$, 
\[
Var(f_k)\leq E\left(Y_k(1)^2\right)\leq C\frac{4^{(2-\alpha)k}}{k}.
\]
\end{claim}
Using \eqref{eq: sum phi n<d} and this claim we obtain.
\begin{lemma}
$\mathrm{Var}\left(V_n(\varphi)\right)\lesssim \frac{n^{2/\alpha}}{\log(n)}$
\end{lemma}
\begin{proof}
For all $k\geq \frac{1}{2\alpha}\log(n)$, $D_k\geq n$ and (up to finitely many $n$) $D_k\leq 4^{k^2}$. Since $\left\{f_k\circ T^j:0\leq j<2\cdot 4^{k^2}\right\}$ is equally distributed as $\left\{Z_k(j):j\leq 1\leq j\leq 2\cdot 4^{k^2}\right\}$, 	we deduce from \eqref{eq: sum phi n<d} and the fact that the $f_k$ are the functions from Corollary \ref{cor: weak Sinai} that;
\begin{itemize}
	\item[(a)] for all $k\geq \frac{1}{2\alpha}\log(n)$, $S_n\left(\varphi_k\right)$ is a sum of independent random variables.  
	\item[(b)]  $S_n\left(\varphi_k\right),\ k\geq \frac{1}{2\alpha}\log(n)$ are independent. 
\end{itemize}

By item (a), 
\begin{align*}
\mathrm{Var}\left(S_n\left(\varphi_k\right)\right)&=\mathrm{Var}(f_k)\left(\sum_{j=0}^{n-2}\frac{(j+1)^2}{D_k^2}+\frac{n^2}{D_k^2}(D_k-n+1)+\sum_{j=1}^{n-1}\frac{(n-j)^2}{D_k^2}\right)\\
&\leq \frac{3n^2}{D_k}\mathrm{Var}(f_k),\ \ \text{as }\ D_k\geq n\\
&\leq 3C n^2\cdot \frac{4^{(2-2\alpha)k}}{k}.
\end{align*}
Here the last inequality follows from Claim \ref{clm: varf} and $\frac{4^{(2-\alpha)k}}{D_k}=4^{(2-2\alpha)k}$. 

Finally by item (b), 
\begin{align*}
\mathrm{Var}\left(V_n(\varphi)\right)&=\sum_{k=\frac{1}{2\alpha}\log(n)+1}^\infty	\mathrm{Var}\left(S_n\left(\varphi_k\right)\right)\\
&\leq 3Cn^2\sum_{k=\frac{1}{2\alpha}\log(n)+1}^\infty \frac{4^{(2-2\alpha)k}}{k}\lesssim \frac{n^{2/\alpha}}{\log(n)}.	
\end{align*}	
\end{proof}
Applying Markov's inequality we obtain. 
\begin{corollary}\label{cor: a>>1 phi}
$\frac{V_n(\varphi)}{n^{1/\alpha}}\xrightarrow[n\to\infty]{}0$ in probability. 
\end{corollary}
We now turn to show that $	n^{-1/\alpha}S_n^{(\mathbf{S})}(h)$ tends to $0$ in probability. The first step is the following simple claim. Recall the notation, $F_k=f_k-\int f_kdm$. 
\begin{claim}\label{clm: a step for small a>1}
For every $k\leq \frac{1}{2\alpha}\log(n)$, 
\[
S_n(h_k)=\sum_{j=0}^{D_k-2}\left(\frac{D_k-j-1}{D_k}\right)F_k\circ T^j-U^n\left(\sum_{j=0}^{D_k-2}\left(\frac{D_k-j-1}{D_k}\right)F_k\circ T^j\right).
\]
\end{claim}
\begin{proof}
As $D_k\leq n$ and $h_k=f_k-\varphi_k$, it follows from \eqref{eq: sum phi n>d} that
\begin{align*}
S_n(h_k)&=\sum_{j=0}^{D_k-2}\left(\frac{D_k-j-1}{D_k}\right)f_k\circ T^j-U^n\left(\sum_{j=0}^{D_k-2}\left(\frac{D_k-j-1}{D_k}\right)f_k\circ T^j\right)\\
&=\sum_{j=0}^{D_k-2}\left(\frac{D_k-j-1}{D_k}\right)F_k\circ T^j-U^n\left(\sum_{j=0}^{D_k-2}\left(\frac{D_k-j-1}{D_k}\right)F_k\circ T^j\right)
\end{align*}
\end{proof}
\begin{lemma}\label{lem: a>1 small}
$	n^{-1/\alpha}S_n^{(\mathbf{S})}(h)$ tends to $0$ in probability.
\end{lemma}
\begin{proof}
By Claim \ref{clm: a step for small a>1}, 
\begin{equation}\label{eq: S_n small a>1}
S_n^{(\mathbf{S})}(h)=A_n-U^n\left(A_n\right),
\end{equation}
where 
\[
A_n:=\sum_{k=0}^{\frac{1}{2\alpha}\log(n)}\hspace{0.4mm}\sum_{j=0}^{D_k-2}\left(\frac{D_k-j-1}{D_k}\right)F_k\circ T^j. 
\]
As $A_n$ is a sum of independent random variables,
\begin{align*}
\mathrm{Var}\left(A_n\right)&=\sum_{k=0}^{\frac{1}{2\alpha}\log(n)}\hspace{0.4mm}\sum_{j=0}^{D_k-2}\left(\frac{D_k-j-1}{D_k}\right)^2\mathrm{Var}\left(F_k\circ T^j\right)\\
&\leq \sum_{k=0}^{\frac{1}{2\alpha}\log(n)}D_k\mathrm{Var}\left(F_k\right).
\end{align*}
Noting that for all $k\in\N$, $\mathrm{Var}(F_k)=\mathrm{Var}(f_k)$, we deduce from the last inequality and Claim \ref{clm: varf} that 
\begin{align*}
	\mathrm{Var}\left(A_n\right)&=\sum_{k=0}^{\frac{1}{2\alpha}\log(n)}D_k\mathrm{Var}\left(f_k\right)\\
	&\lesssim \sum_{k=0}^{\frac{1}{2\alpha}\log(n)} \frac{4^{2k}}{k}\lesssim \frac{n^{2/\alpha}}{\log(n)}.
\end{align*}
Next, as $\int A_ndm=0$, it follows from Chebyshev's inequality that for every $\epsilon>0$,
\begin{align*}
m\left(\left|n^{-1/\alpha}A_n\right|>\epsilon\right)&\leq \frac{\mathrm{Var}\left(A_n\right)}{n^{2/\alpha}\epsilon^2}\lesssim \frac{1}{\log(n)\epsilon^2}.	
\end{align*}	
This shows that $n^{-1/\alpha}A_n$ tends to $0$ in probability.

Since $A_n$ and $U^n\left(A_n\right)$ are equally distributed,  $n^{-1/\alpha}U^n\left(A_n\right)$ also tends to $0$ in probability. The claim now follows from the converging together lemma. 
\end{proof}
\begin{proposition}\label{prop: middle a>1}
$S_n^{(\mathbf{M})}(f)$ converges in distribution to an $S_\alpha(\sigma,1,0)$ random variable with $\sigma^\alpha=\ln2$. 
\end{proposition}
We postpone the proof of this proposition to Subsection \ref{sub: proof of middle a>1} and we can now prove Theorem \ref{skewed CLT for a>1}. 

\begin{proof}[Proof of Theorem \ref{skewed CLT for a>1}]
We deduce from Lemmas \ref{lem: a>1 small},  \ref{lem: a>>1 large} and \ref{lem: a>>1 phi} that 
\[
S_n^{(\mathbf{S})}(h)+V_n(\varphi)+S_n^{(\mathbf{L})}(f)\xrightarrow[n\to\infty]{}0,\ \ \text{in probability.}
\]
The result now follows from \eqref{eq: decomp for a>1}, Corollary \ref{cor: a>>1 phi}, Proposition \ref{prop: middle a>1} and the converging together lemma. 	
\end{proof}	

\subsection{Proof of Proposition \ref{prop: middle a>1}}\label{sub: proof of middle a>1}
The proof of this proposition goes along similar lines as the proof of Proposition \ref{prop: a<1 middle h} with some (rather) obvious modifications. We first define, 
\begin{align*}
S_n^{(\mathbf{M})}(Z)&:=\sum_{k=\frac{1}{2\alpha}\log(n)+1}^{\frac{1}{\alpha}\log(n)}S_{n}\left(Z_k(\cdot)\right),\\
S_n^{(\mathbf{M})}(Y)&:=\sum_{k=\frac{1}{2\alpha}\log(n)+1}^{\frac{1}{\alpha}\log(n)}S_{n}\left(Y_k(\cdot)\right)\\
S_n^{(\mathbf{M})}(X)&:=\sum_{k=\frac{1}{2\alpha}\log(n)+1}^{\frac{1}{\alpha}\log(n)}S_{n}\left(X_k(\cdot)\right)
\end{align*}	
The following is the analogue of Lemma \ref{lem: a<1 equal dist} for the current case. 
\begin{lemma} \label{lem: a>1 equal dist skewed}
	$S_n^{(\mathbf{M})}(f)$ and $S_n^{(\mathbf{M})}(Z)+B_n$ are equally distributed. 
\end{lemma}
The proof of Lemma \ref{lem: a>1 equal dist skewed} is similar to the proof of Lemma \ref{lem: a<1 equal dist} with obvious modifications. We leave it to the reader. Proposition \ref{prop: middle a>1} follows from Lemma \ref{lem: a>1 equal dist skewed} and the following. 

\begin{lemma}\label{lem: a>1 middle Z skewed}
	$ $
	\begin{itemize}
		\item[(a)] $
		\frac{1}{n^{1/\alpha}}\left(S_n^{(\mathbf{M})}(X)-S_n^{(\mathbf{M})}(Z)-B_n\right)\xrightarrow[n\to\infty]{}0$ in measure. 
		\item[(b)] $\frac{1}{n^{1/\alpha}}S_n^{(\mathbf{M})}(X)$ converges in distribution to an $S_\alpha\left(\sqrt[\alpha]{\ln(2)},1,0\right)$ random variable.
	\end{itemize}
	Consequently  $S_n^{(\mathbf{M})}(Z)+B_n$ converges in distribution to a $\aS\left(\sqrt[\alpha]{\ln(2)}\right)$ random variable
\end{lemma}

\begin{proof}[Proof of Lemma \ref{lem: a>1 middle Z skewed}.(b)]
For all $n\in\mathbb{N}$, $S_n^{(\mathbf{M})}(X)$ is a sum of independent totally skewed $\alpha$ stable random variables. By \cite[Property 1.2.1]{MR1280932} $n^{-1/\alpha}S_n(X)$ is $S_\alpha(\Sigma_n,1,0)$ distributed with
\begin{align*}
(\Sigma_n)^\alpha&:=\frac{1}{n}\sum_{k=\frac{1}{2\alpha}\log(n)+1}^{\frac{1}{\alpha}\log(n)}\frac{n}{k}\sim \ln(2),\ \ \text{as}\ n\to\infty. 
\end{align*}
The result follows from the fact that if $B_n$ is $S_\alpha(\Sigma_n,1,0)$ distributed and $\Sigma_n\to \sqrt[\alpha]{\ln(2)}$ then $B_n$ converges to an $S_\alpha(\sqrt[\alpha]{\ln(2)},1,0)$ random variable.
\end{proof}
\begin{proof}[Proof of Lemma \ref{lem: a>1 middle Z skewed}.(a)]
We assume $n$ is large enough so that if $k>\frac{1}{2\alpha}\log(n)$ then $n<4^{k^2}$. 

 Firstly, since for all $k\in\N$ and $j\leq d_k$, 
	\[
	0<Y_k(j)-Z_k(j)\leq 4^{-{k}},
	\]
	we have
	\begin{align*}
	\left|S_n^{(\mathbf{M})}(Y)-S_n^{(\mathbf{M})}(Z)\right|&\leq \sum_{k=\frac{1}{2\alpha}\log(n)+1}^{\frac{1}{\alpha}\log(n)}S_n\left(Y_k(\cdot)-Z_k\left(\cdot\right)\right)\\
	&\leq \sum_{k=\frac{1}{2\alpha}\log(n)+1}^{\frac{1}{\alpha}\log(n)}4^{-k}n\lesssim n^{1-\frac{1}{\alpha}}
	\end{align*}
	Consequently,
	\begin{equation}\label{eq: a>1 skewed S_n(Y)-S_n(Z)}
	\frac{1}{n^{1/\alpha}}\left| S_n^{(\mathbf{M})}(Y)-S_n^{(\mathbf{M})}(Z)\right|\lesssim n^{1-\frac{2}{\alpha}}\xrightarrow[n\to\infty]{}0.
	\end{equation}
	We turn to look at $S_n^{(\mathbf{M})}(X)-S_n^{(\mathbf{M})}(Y)-B_n$. For all $n$, 
	\[
	S_n^{(\mathbf{M})}(X)-\W_n^{(\mathbf{M})}(Y)-B_n=\mathrm{II}_n+\mathrm{III}_n,
	\]
	where 
	\begin{align*}
	\widetilde{V}_k(m)&:=X_k(m)1_{\left[X_k(m)>4^k\right]}\ \ \ \ \text{and}\\ 
	\widehat{V}(m)&:=\sum_{k=\frac{1}{2\alpha}\log(n)+1}^{\frac{1}{\alpha}\log(n)}\left(X_k(m)1_{\left[X_k(m)\leq 2^k\right]}-\mathbb{E}\left(X_k(m)1_{\left[X_k(m)\leq 2^k\right]}\right)\right), 
	\end{align*}
	and
	\begin{align*}
	\mathrm{II}_n&:=n^{-1/\alpha}S_{n}\left(\widehat{V}\right)\\
	\mathrm{III}_n&:=n^{-1/\alpha}\sum_{k=\frac{1}{2\alpha}\log(n)+1}^{\frac{1}{\alpha}\log(n)}S_{n}(\widetilde{V}_k).
	\end{align*}
	Similarly to the proof of Lemma \ref{lem: a<1 Large}, 
	\begin{align*}
	\P\left(\mathrm{III}_n\neq 0\right)&\leq \sum_{k=\frac{1}{2\alpha}\log(n)+1}^{\frac{1}{\alpha}\log(n)}\sum_{j=1}^n\P\left(\widetilde{V}_k(j)\neq 0\right)\\
	&\leq \sum_{k=\frac{1}{2\alpha}\log(n)+1}^{\frac{1}{\alpha}\log(n)}\sum_{j=1}^n\P\left(X_k(m)>4^k\right)\\
	&\lesssim 2Cn\sum_{k=\frac{1}{2\alpha}\log(n)+1}^{\frac{1}{\alpha}\log(n)}\frac{4^{-k\alpha}}{k}\lesssim \frac{1}{\log(n)},
	\end{align*}
	Now $ \widehat{V}(m)$, $1\leq m \leq n$, are zero-mean, independent random variables. By Proposition \ref{prop: second moment for a>=1}, they also have second moment and for all $1\leq j\leq n$, 
	\begin{align}
	\E\left(\widehat{V}(j)^2\right)&\lesssim 2\sum_{k=\frac{1}{2\alpha}\log(n)+1}^{\frac{1}{\alpha}\log(n)}\E\left(X_k(m)^21_{\left[X_k(m)\leq 2^k\right]}\right) \nonumber \\
	&\lesssim  \sum_{k=\frac{1}{2\alpha}\log(n)+1}^{\frac{1}{\alpha}\log(n)}\frac{2^{(2-\alpha)k}}{k}\lesssim \frac{n^{\frac{2}{\alpha}-1}}{\log(n)}.
	\end{align}
	It follows from Markov's inequality that for every $\epsilon>0$,
	\begin{align*}
	\P\left(\big|\mathrm{II}_n\big|>\epsilon\right)
	&\leq \epsilon^{-2}n^{-2/\alpha}\E\left(\left|S_n\left(\widehat{V}\right)\right|^2\right)\\
	&= \epsilon^{-2}n^{-2/\alpha}\left(n\E\left(\widehat{V}(1)^2\right)\right)\lesssim \frac{1}{\epsilon^{2}\log(n)}.
	\end{align*}
	This concludes the proof that 
$\mathrm{II}_n\xrightarrow[n\to\infty]{}0$ in probability. 

	The claim now follows from \eqref{eq: a>1 skewed S_n(Y)-S_n(Z)} and the convergence in probability of $\mathrm{II}_n+\mathrm{III}_n$ to $0$. 
\end{proof}	

\subsection{Deducing Theorem \ref{thm: CLT a>1} from Theorem \ref{skewed CLT for a>1}}
This is similar to the strategy and steps which were carried out in Subsection \ref{sub: Theorem general from skewed}. 

Recall the notation $\hat{\varphi}_k:=\frac{1}{D_k}\sum_{j=0}^{D_k-1}g_k\circ T^j$ and $\hat{h}_k:=g_k-\hat{\varphi}_k$. Since for all $k\in\N$, $\varphi_k$ and $\hat{\varphi}_k$ are equally distributed, by mimicking the proof of Lemma \ref{lem: a>>1 phi} and Corollary \ref{cor: a>>1 phi} we obtain the following. 
 \begin{lemma}\label{lem: a>>1 phi'}
	$n^{-1/\alpha}V_n\left(\hat{\varphi}\right)$ converges to $0$ in probability.  
\end{lemma}
Next we have the following analogue of Lemma \ref{lem: beta general 1}.

\begin{lemma}\label{lem: beta general 1 a>1}
	$ $
	\begin{itemize}
		\item[(a)] $\frac{1}{n^{1/\alpha}}\left(S_n^{(\mathbf{S})}(h),S_n^{(\mathbf{S})}\big(\hat{h}\big)\right)$ converges in probability to $(0,0)$. 
		\item[(b)] $\frac{1}{n^{1/\alpha}}\left(S_n^{(\mathbf{M})}(f)+B_n,S_n^{(\mathbf{M})}(g)+B_n\right)$ converges in distribution to $(W,W')$ where $W,W'$ are independent $S_\alpha\left(\sqrt[\alpha]{\ln(2)},1,0\right)$ random variables. 
		\item[(c)]$\frac{1}{n^{1/\alpha}}\left(S_n^{(\mathbf{L})}(f),S_n^{(\mathbf{L})}(g)\right)$ converges in probability to $(0,0)$. 
	\end{itemize}
\end{lemma}
\begin{proof}
As for all $k$, $h_k$ and $\hat{h}_k$ are equally distributed, by mimicking the proof of Lemma \ref{lem: a>1 small} one proves that 
\[
\frac{1}{n^{1/\alpha}}S_n^{(\mathbf{S})}\left(\hat{h}\right)\xrightarrow[n\to\infty]{}0\ \text{in probability}.
\]
Part (a) follows from this and Lemma \ref{lem: a>1 small}.

The deduction of part (c) from Lemma \ref{lem: a>>1 large} and its proof is similar. 

Part (b) follows from Proposition \ref{prop: middle a>1} as $S_n^{(\mathbf{M})}(f)$ and $S_n^{(\mathbf{M})}(g)$ are independent and equally distributed. 
\end{proof}
Now fix $\beta\in [-1,1]$ and recall that  $H=\Phi_{\beta}\left(h,\hat{h}\right)$ where $\Phi_\beta$ is the linear function defined by for all $x,y\in\R$, 
\[
\Phi_\beta(x,y):=\left(\frac{\beta+1}{2}\right)^{1/\alpha}x-\left(\frac{1-\beta}{2}\right)^{1/\alpha}y
\]
\begin{proof}[Proof of Theorem \ref{thm: CLT a>1}]
Writing, 
\[
A_n:=\frac{1}{n^{1/\alpha}}\left(S_n^{(\mathbf{S})}(h)+V_n(\varphi)+S_n^{(\mathbf{L})}(f),S_n^{(\mathbf{S})}\big(\hat{h}\big)+V_n\big(\hat{\varphi}\big)+S_n^{(\mathbf{L})}(g)\right),
\]	
we have for all $n\in \N$,
\begin{equation}\label{eq: a>>1 final}
S_n(H)=\Phi_\beta\big(A_n\big)+\Phi_\beta\left(n^{-1/\alpha}S_n^{(\mathbf{M})}(f), n^{-1/\alpha}S_n^{(\mathbf{M})}(g)\right).
\end{equation}
By Lemma \ref{lem: a>>1 phi'} and parts parts (a) and (c) of Lemma \ref{lem: beta general 1 a>1} , $A_n\to (0,0)$ in probability. Since $\Phi_\beta$ is continuous with $\Phi_\beta(0,0)=0$, it follows that  $\Phi_\beta\big(A_n\big)$ converges to $0$ in probability as $n\to\infty$. 

By Lemma \ref{lem: beta general 1 a>1}.(b) and the continuous mapping theorem, 
\[
\Phi_\beta\left(n^{-1/\alpha}S_n^{(\mathbf{M})}(f), n^{-1/\alpha}S_n^{(\mathbf{M})}(g)\right)\Rightarrow^d \Phi_\beta(W,W'),
\]
where $W,W'$ are independent $S_\alpha(\sqrt[\alpha]{\ln(2)},1,0)$ distributed random variables. By \cite[Property 1.2.13]{MR1280932}, $\Phi_\beta(W,W')$ is $S_\alpha(\sqrt[\alpha]{\ln(2)},\beta,0)$ distributed. 

The conclusion now follows from \eqref{eq: a>>1 final} and the converging together lemma.
\end{proof}	
\section{Appendix}
The following tail bound follows easily from  \cite[Property 1.2.15]{MR1280932}
\begin{proposition}\label{prop: ST}
	There exists $C>0$ such that if $Y$ is $S_\alpha(\sigma,1,0)$ distributed with $0<\sigma\leq1$ and $K>1$ then 
	\[
	\P\left(Y>K\right)\leq C\sigma^\alpha K^{-\alpha}.
	\]
\end{proposition}
In a similar way to the appendix in \cite{2022arXiv221103448K},the tail bound implies the following two estimates on moments of truncated $S_\alpha(\sigma,1,0)$ random variables. 
\begin{corollary}\label{cor: moments}
	For every $r>\alpha$, there exists $C>0$ such that if $Y$ is $S_\alpha(\sigma,1,0)$ distributed with $0<\sigma\leq1$ and $K>1$,
	\[
	\mathbb{E}\left(Y^r1_{[0\leq Y\leq K]}\right)\leq C\sigma^\alpha K^{r-\alpha}. 
	\] 
\end{corollary}
\begin{proof}
	The bound follows from
	\begin{align*}
	\mathbb{E}\left(Y^r1_{[0\leq Y\leq K]}\right)&=\int_{0}^\infty \P\left(Y^r1_{[0\leq Y\leq K]}>t\right)dt\\
	&=	\int_{0}^{K^r} \P\left(Y>t^{1/r}\right)dt\\
	&=r\int_{0}^{K} u^{r-1}\P\left(Y>u\right)du\\
	&=r\int_{0}^{1} u^{r-1}\P\left(Y>u\right)du+r\int_{1}^{K} u^{r-1}\P\left(Y>u\right)du\\
	&\leq r+C\sigma^\alpha r\int_{1}^Ku^{r-1-\alpha}du. 
\end{align*}
Here the last inequality follows from Proposition \ref{prop: ST}. 
\end{proof}
\begin{corollary}\label{cor: moments below a}
For every $r<\alpha$, there exists $C>0$ such that if $Y$ is $S_\alpha(\sigma,1,0)$ distributed with $0<\sigma\leq1$as $K\to\infty$,
\[
\mathbb{E}\left(Y^r1_{[Y\geq K]}\right)\lesssim C\sigma^\alpha K^{r-\alpha}. 
\] 
\end{corollary}
The proof of Corollary \ref{cor: moments below a} is similar to the proof of Corollary \ref{cor: moments}. The following is important in the proofs of Theorems \ref{thm: a>1} and \ref{thm: CLT a>1}.
\begin{proposition}\label{prop: second moment for a>=1}
For every $K,\sigma>0$, if $X$ is $S_\alpha(\sigma,1,0)$ distributed then $X1_{X<K}$ is square integrable. Furthermore,	there exists $C>0$ such that for every $X$ an $S_\alpha(\sigma,1,0)$ random variable with $0<\sigma\leq 1$ and $K>1$, 
\[
\E\left(\left(X \mathbf{1}_{[X<K]}\right)^2\right)\leq C\sigma^\alpha K^{2-\alpha}.
\]   
\end{proposition}
\begin{proof}
 Let $Y$ be an $S_\alpha(1,1,0)$ random variable and note that $\sigma Y$ and $X$ are equally distributed. By \cite[Theorems 2.5.3 and 2.5.4]{MR0854867} see also equations (1.2.11) and (1.2.12) in \cite{MR1280932}, $\P(Y<-\lambda)$ decays faster than any polynomial as $\lambda\to\infty$.  This implies that $Y1_{Y<0}$ has moments of all orders and
 \[
 \E\left(\left(X1_{[X<0]}\right)^2\right)=\sigma^2\E\left(Y^21_{[Y<0]}\right)\leq D,
 \]
 where $D=\E\left(Y^21_{[Y<0]}\right)$. 
Now by this and Corollary \ref{cor: moments}, we have
\begin{align*}
\E\left(\left(X \mathbf{1}_{[X<K]}\right)^2\right)&\leq 4\left(\E\left(\left(X \mathbf{1}_{[0<X<K]}\right)^2\right)+\E\left(\left(X \mathbf{1}_{[X<0]}\right)^2\right)\right)\\
&\leq 4\left(C\sigma^\alpha K^{2-\alpha}+D\right)\sim 4C\sigma^\alpha K^{2-\alpha},\ \ \text{as}\ K\to\infty. 
\end{align*}
The claim follows from this upper bound. 
\end{proof}


\bibliographystyle{abbrv}
\bibliography{embedding.bib}

\end{document}